\documentclass[11pt]{amsart}

\usepackage[T1]{fontenc}
\usepackage[latin1]{inputenc}
\usepackage{amsmath,amsfonts,amssymb,mathrsfs,amsthm}
\usepackage{mathrsfs}
\usepackage{xpatch}
\usepackage{color}
\usepackage[pdftex]{hyperref}
\usepackage{graphicx,functan}
\usepackage{framed,fancybox}
\usepackage{bbm}

\DeclareRobustCommand\slab{s}
\DeclareRobustCommand\nulab{\nu}

\theoremstyle{plain}
\newtheorem{theorem}{Theorem}[section]
\newtheorem{lemma}[theorem]{Lemma}
\newtheorem{proposition}[theorem]{Proposition}
\newtheorem{corollary}[theorem]{Corollary}
\theoremstyle{definition}
\newtheorem{remark}[theorem]{Remark}
\DeclareRobustCommand\slab{s}
\DeclareRobustCommand\nulab{\nu}

\newcommand{\N}{\mathbb{N}}
\newcommand{\R}{\mathbb{R}}
\newcommand{\K}{\mathbb{K}}
\newcommand{\bs}{\boldsymbol}
\newcommand{\eps}{\varepsilon}
\newcommand{\lraup}{\relbar\joinrel\rightharpoonup}
\newcommand{\bsl}{\pmb{\langle}}
\newcommand{\bsr}{\pmb{\rangle}}

\newcommand{\opL}{\mathscr{L}}

\newcommand{\ked}{\widehat{k}_{\eps}}
\def\MR#1{\href{http://www.ams.org/mathscinet-getitem?mr=#1}{MR#1}}

\numberwithin{equation}{section}

\begin{document}
	
	\title[]{Nonlocal Lagrange multipliers and transport densities} 
	
	\author[]{Assis Azevedo}
	\address{CMAT and Departamento de Matem\'atica, Escola de Ci\^encias, Universidade do Minho, Campus de Gualtar, 4710-057 Braga, Portugal}
	\email{assis@math.uminho.pt}	
	\author[]{Jos\'e Francisco Rodrigues}
	\address{CMAFcIO -- Departamento de Matem\'atica, Faculdade de Ci\^encias, Universidade de Lisboa
		P-1749-016 Lisboa, Portugal}
	\email{jfrodrigues@ciencias.ulisboa.pt}
	\author[]{Lisa Santos} 
	\address{CMAT and Departamento de Matem\'atica, Escola de Ci\^encias, Universidade do Minho, Campus de Gualtar, 4710-057 Braga, Portugal}
	\email{lisa@math.uminho.pt}
	
	
	\begin{abstract}	
		We prove the existence of generalised solutions of the Monge-Kantorovich equations  with fractional $s$-gradient constraint, $0<s<1$, associated to a general, possibly degenerate, linear fractional operator of the type, 
		\begin{equation*}
			\mathscr L^su=-D^s\cdot(AD^su+\bs bu)+\bs d\cdot D^su+cu ,
		\end{equation*}
		with integrable data, in the space $\Lambda^{s,p}_0(\Omega)$, which is the completion of the set of smooth functions with compact support in a bounded domain $\Omega$ for the $L^p$-norm of the distributional Riesz fractional gradient $D^s$ in $\R^d$ (when $s=1$, $D^1=D$ is the classical gradient). The transport densities arise as generalised Lagrange multipliers in the dual space of $L^\infty(\R^d)$ and are associated to the variational inequalities of the corresponding transport potentials under the constraint $|D^su|\leq g$. Their existence is shown by approximating the variational inequality through a penalisation of the constraint and nonlinear regularisation of the linear operator $\mathscr L^su$. For this purpose, we also develop some relevant properties of the spaces $\Lambda^{s,p}_0(\Omega)$, including the limit case $p=\infty$ and the continuous embeddings $\Lambda^{s,q}_0(\Omega)\subset \Lambda^{s,p}_0(\Omega)$, for $1\le p\le q\le\infty$. We also show the localisation of the nonlocal problems ($0<s<1$), to the local limit problem with classical gradient constraint when $s\rightarrow1$, for which most results are also new for a general, possibly degenerate, partial differential operator $\mathscr L^1u$ only with integrable coefficients and bounded gradient constraint.
	\end{abstract}
	
	\maketitle
	

	\section{Introduction}	
	In a bounded open set $\Omega$ of $\R^d$, consider the  model problem for the pair of functions $(u,\lambda)$,
	\begin{align}
		&	-D\cdot\big((\delta+\lambda)Du\big)=f\ \text{ in $\Omega$}, \ u=0\ \text{ on } \partial\Omega\label{Int_1} \\
		&|Du|\le 1,\quad\lambda\ge0,\quad \lambda(|Du|-1)=0\ \text{ in $\Omega$,}\label{Int_2}
	\end{align}
	where $\delta\geq 0$ is a constant, $D$ denotes the gradient, $D\cdot$ denotes the divergence and $f=f(x)$ is a given function.
	
	For $\delta>0$, the problem \eqref{Int_1}--\eqref{Int_2}, being equivalent to minimise the functional
	\begin{equation}
		\label{Int_3}
		u\mapsto\tfrac{\delta}{2}\int_\Omega|Du|^2-\int_\Omega fu 
	\end{equation}
	in the convex subset of $H^1_0(\Omega)$ subjected to the constraint $|Du|\leq 1$ in $\Omega$, is well-known to model the elastoplastic torsion of a cylindric bar of cross section $\Omega$, where $\lambda$ is the respective Lagrange multiplier. In $1972$, Br\'ezis \cite{Brezis72} has shown that, if $f=const>0$ and $\Omega$ is simply connected, $\lambda\in L^\infty(\Omega)$ is unique and even continuous if $\Omega$ is convex. This was partially extended to more general strictly convex functionals than \eqref{Int_3}, by Chiad\`o Piat and Percivale \cite{ChiadoPiatPercivale1994}, for $f\in L^p(\Omega)$, $p>d$, obtaining a solution $u$ in $C^{1,\alpha}(\overline{\Omega})$, $\alpha=1-\tfrac{d}{p}$ and $\lambda$ as a positive Radon measure (see the survey \cite{RodriguesSantos2019}, for references and more results).
	
	In the degenerate case $\delta=0$, \eqref{Int_1}--\eqref{Int_2} are usually called the Monge-Kantorovich equations, as they appear in a classical mass transfer problem \cite{EG}, where $u$ and $\lambda$ represent the transport potential and density, respectively. This is the dual problem of \eqref{Int_3} with $\delta=0$ over all Lipschitz continuous functions with $|Du|\leq 1$ and vanishing on $\partial\Omega$. This same problem also arises in shape optimization \cite{BBS},  in the equilibrium configurations \cite{BB22} and in the time discretisation of the growing sandpile problem \cite{DI}. 
	
	In general, and specially in the case $\delta=0$ with more general gradients thresholds, the main difficulty in studying  \eqref{Int_1}--\eqref{Int_2}  is the non-regularity of the flux, since $Du$ is just bounded and it can not be multiplied by
	$\lambda$, whenever this is a Radon measure. Several approaches have been proposed, by relaxing the Monge-Kantorovich problem (see \cite{BBS}, \cite{I2009} or \cite{BBdp}).
	
	A different and more direct approach was proposed by \cite{AzevedoSantos2017} to solve \eqref{Int_1}--\eqref{Int_2} with $\delta\geq 0$, $f\in L^2(\Omega)$ and a variable general constraint $|Du|\leq g\in L^\infty(\Omega)$, with $g>0$, by proving the existence of a pair $(u,\lambda)\in W^{1,\infty}(\Omega)\times L^\infty(\Omega)'$.
	The generalised Lagrange multiplier $\lambda$ being a charge, i.e., an element of $ L^\infty(\Omega)'$, allows to interpret the equation \eqref{Int_1} in a duality sense and the second and third conditions of \eqref{Int_2} (with $1$ replaced by $g$) in the dual space $L^\infty(\Omega)'$.
	
	Recently, this charges approach was extended in \cite{AzevedoRodriguesSantos2022} to a class of coercive nonlocal problems considered in  \cite{RodriguesSantos2019} with fractional gradient constraint of the type
	\begin{equation}
		\label{Int_4}
		|D^su|\leq g, \quad 0<s<1,
	\end{equation}
	where $D^s$ is the distributional fractional Riesz gradient. The fractional $s$-gradient $D^s$ has been recently studied by several authors \cite{ShiehSpector2015}, \cite{S2020}, \cite{ComiStefani2019}, \cite{ComiStefani2019_2}. It may be defined via smooth functions $C^\infty_c(\R^d)$ by the convolution of the classical gradient with the Riesz kernel $I_{1-s}$, i.e., $D^su=I_{1-s}*Du=D(I_{1-s}*u)$, with the nice properties $(-\Delta)^su=-D^s\cdot (D^su)$ and
	\begin{equation}
		\label{Int_5}\int_{\R^d}u D^s\cdot\bs\xi=-\int_{\R^d} D^su\cdot\bs\xi,\quad\forall \bs\xi\in C^\infty_c(\R^d)^d,
	\end{equation}
	where $D^s\cdot$ denotes the $s$-divergence and $(-\Delta)^s$ the fractional $s$-Laplacian. For smooth functions with compact support $D^s$ can also be equivalently defined by a vector-valued fractional singular integral, which satisfies elementary physical requirements, such as translational and rotational invariances, homogeneity of degree $s$ under isotropic scaling and certain basic continuity properties \cite{S2020}, in order to model long-range forces and nonlocal effects in continuum mechanics.
	
	Another important property of $D^s$ is due to the fact that the Riesz kernel $I_{1-s}$ approaches the identity operator as $s\rightarrow 1$, which implies that $D^su\longrightarrow Du$ in $L^p$-spaces, provided $Du\in \bs L^p(\R^d)=L^p(\R^d)^d$ (see Section \ref{section2}, for details). However it should be noted that even when $u$ has compact support in $\R^d$ and $D^su$ makes sense as a $p$-integrable function, in general, $D^su$ has not compact support in contrast with $Du=D^1u$.
	
	Here we shall be concerned with the more general fractional Monge-Kantorovich-type problem for a function $u$, satisfying $u=0$ in $\R^d\setminus\Omega$, and a charge $\lambda$, such that
	\refstepcounter{equation}
	\begin{align}
		&\mathscr L^su-D^s\cdot(\lambda D^su)=f-D^s\cdot \bs f
		\label{Int_6}	\tag*{(\theequation)$_{\slab}$}\\
		&			\refstepcounter{equation}
		|D^su|\le g_s,\quad\lambda\ge0\quad \text{ and }\quad \lambda(|D^su|-g_s)=0. \label{Int_7}	\tag*{(\theequation)$_{\slab}$}
	\end{align}
	
	For a bounded positive threshold $g_s$, the first condition in \ref{Int_7} holds a.e. $x\in\R^d$, for $0<s<1$, and a.e. in $\Omega$, for $s=1$, while the second and third ones are interpreted in $L^\infty(\R^d)'$ and in $L^\infty(\Omega)'$, respectively.
	
	The equation \ref{Int_6} must be interpreted in an appropriate functional space duality with the bilinear form associated to a linear operator for $0<s\leq 1$, possibly degenerate, in the general form:
	\refstepcounter{equation}\begin{equation} \mathscr L^su=-D^s\cdot(AD^su+\bs bu)+\bs d\cdot D^su+cu
		\label{Int_8}	\tag*{(\theequation)$_{\slab}$},
	\end{equation}
	where the nonnegative matrix $A=A(x)$ has integrable coefficients, which may degenerate or even vanish completely, the vector fields $\bs b$ and $\bs d$, as well as the function $c$ and the given data $f$ and $\bs f$ are also merely integrable in the case of bounded $g_s$, even in the classical local case $s=1$.
	
	The fractional setting for the homogeneous Dirichlet condition is considered within the functional framework of the following family of Banach spaces
	\begin{equation}
		\label{Int_9}
		\Lambda^{s,q}_0(\Omega)\subset \Lambda^{s,p}_0(\Omega),\quad 1\leq p\leq q\leq \infty, \ 0<s<1,
	\end{equation}
	where $\Lambda^{s,2}_0(\Omega)$ are the usual fractional Sobolev spaces $H^s_0(\Omega)$ and the limit case $s=1$ corresponds to the usual Sobolev spaces $H^1_0(\Omega)$ and  $W^{1,p}_0(\Omega)$ if $p\neq 2$. For $1\leq p<\infty$, $\Lambda^{s,p}_0(\Omega)$ is the closure of $C^\infty_0(\Omega)$ for the norm $\|D^su\|_{\bs L^p(\R^d)}$ (see Section \ref{section2}).
	
	We observe that, for $p\neq 2$, $0<s<1$, the Lions-Cald\'eron spaces $\Lambda^{s,p}_0(\Omega)$ are different from the Sobolev-Slobodeckij spaces $W^{s,p}_0(\Omega)$, although they are contiguous (see \cite[p 219]{Adams1975} or \cite{Campos2021}), i.e. 
	\begin{equation*}
		\Lambda^{s+\eps,p}_0(\Omega)\subsetneqq W^{s,p}_0(\Omega)\subsetneqq\Lambda^{s-\eps,p}_0(\Omega),\quad s>\eps>0, \ 1<p<\infty,\ p\neq 2.
	\end{equation*}
	
	The paper is organised as follows: in Section \ref{section2} we develop the required functional framework for the Riesz fractional derivatives and we recall and prove some interesting properties of the spaces $\Lambda^{s,p}_0(\Omega)$, including \eqref{Int_9}; in Section \ref{section3}, we precise the assumptions on $\mathscr L^s$, which may be a degenerate operator, and we prove the existence of a solution to the corresponding pseudo-monotone variational inequality with the convex set of the $s$-gradient constraint \eqref{Int_4} in $H^s_0(\Omega)$ and in $\Lambda^{s,\infty}_0(\Omega)$ for nonnegative threshold $g\in L^2_{loc}(\R^d)$ and $g\in L^\infty_{loc}(\R^d)$, respectively. We also give sufficient conditions for the operator $\mathscr L^s$ to be strictly coercive in $H^s_0(\Omega)$ and, as a consequence, we extend the strong continuous dependence (and the uniqueness) of the transport potential $u$ with respect to the data, including the continuous dependence on the $s$-gradient thresholds.
	
	Our main results are in Section \ref{section4}, where we prove the existence of a generalised transport potential-density pair solving the Monge-Kantorovich equations \ref{Int_6} and \ref{Int_7} under rather general conditions on the operator $\mathscr L^s$, including the $L^1$ integrability of its coefficients. The proof is based on a new generalised weak continuous dependence on the pair $(u, \lambda)$ with respect not only on the coefficients of  $\mathscr L^s$ and on the data $f$, $\bs f$ (in $L^1$) but also on the threshold $g$ (in $L^\infty$) and on the solvability and {\em a  priori} estimates of a suitable family of approximation problems in the space $ \Lambda^{s,q}_0(\Omega)$, for a large  finite $q$, with a penalisation of the $s$-gradient and with a nonlinear regularisation of $q$-power type of the possible degenerate operator $\mathscr L^s$.
	Finally in Section \ref{section5} we extend the weak convergence on the generalised localisation of the transport potentials and densities as the fractional parameter $s\rightarrow 1$, improving the result of \cite{AzevedoRodriguesSantos2022}.
	In Sections 4 and 5, we work with generalised sequences, also called  nets, see for instance \cite{KantorovichAkilov1982}.

	\section{The functional framework}\label{section2}

	Following \cite{ShiehSpector2015} we recall that the fractional gradient of order $s\in (0,1)$, denoted by $D^s=\big(D^s_1,\ldots,D^s_d\big)$, may be defined in the distributional sense by 
	\begin{equation*} 
		D^s u= D(I_{1-s}u) 
	\end{equation*}
	for any function $u\in L^p(\R^d)$, $1<p<\infty$, such that the Riesz potential $I_{1-s}u=I_{1-s}* u$ is locally integrable, i.e., for each $i=1,\ldots,d$: 
	\begin{equation}\label{Ds}
		\big\langle D^s_iu, \varphi\big\rangle=-
		\big\langle I_{1-s}*u,D_i\varphi  \big\rangle=
		\int_{\R^d}(I_{1-s}*u)D_i\varphi, \qquad \forall \varphi\in C^\infty_c(\R^d).
	\end{equation}
	
	The Riesz kernel of order $\alpha\in(0,1)$, for $x\in\R^d\setminus\{0\}$, is given by
	\begin{equation*}
		I_\alpha(x)=\frac{\gamma_{d,\alpha}}{|x|^{d-\alpha}},\qquad\text{with }\gamma_{d,\alpha}=\frac{\Gamma(\frac{d-\alpha}2)}{\pi^\frac{d}22^\alpha\Gamma(\frac{\alpha}{2})},
	\end{equation*}
	and it satisfies the following well-known properties which proof is reproduced for completeness. 
	
	We fix the notation $B(x,r)$ for the open ball centered at $x\in\R^d$ and radius $r>0$.
	
	\begin{lemma} \label{I_II_II} Let $I_\alpha$ be the Riesz kernel, $0<\alpha<1$, $p\in(1,\infty)$ and $R>0$. Then, denoting by $\sigma_{d-1}$ the surface area of the unit sphere in $\R^d$, we have:
		\begin{enumerate}
			\item[(i)] $\|I_\alpha\|_{L^1(B(0,R))}=\sigma_{d-1}\frac{\gamma_{d,\alpha}}{\alpha}R^\alpha$;		
			\item[(ii)] If $\alpha p<d$ then $\|I_\alpha\|_{L^{p'}(\R^d\setminus B(0,R))}=\gamma_{d,\alpha}\Big(\sigma_{d-1}\frac{p-1}{d-\alpha p}\Big)^{\frac{1}{p'}}R^\frac{\alpha p-d}{p}$.
		\end{enumerate}
		
		As a consequence $\displaystyle\lim_{\alpha\rightarrow 0}\|I_\alpha\|_{L^1(B(0,R))}=1$ and $\displaystyle\lim_{\alpha\rightarrow 0}\|I_\alpha\|_{L^{p'}(\R^d\setminus B(0,R))}=0$.
	\end{lemma}
	
	\begin{proof} 
		We start by noticing that, if $b\neq d$, $0\leq R_1\leq R_2<+\infty$, then
		\begin{equation*}
			\int_{B(0,R_2)\setminus B(0,R_1)}\frac{1}{|x|^b}dx  =  \sigma_{d-1}\int_{R_1}^{R_2} r^{d-1-b}dr = \sigma_{d-1}\bigg[\frac{R_2^{d-b}}{d-b}-\frac{R_1^{d-b}}{d-b}\bigg].
		\end{equation*}
		Considering first $b=d-\alpha$, $R_2=R$ and $R_1=0$ we obtain (i). Then choosing $b=(d-a)\frac{p}{p-1}$, $R_1=R$ we obtain (ii) by letting $R_2\rightarrow\infty$ and noticing that $a p<d$ or equivalently $d-b<0$.
		
		Since
		\begin{equation*}
			\lim_{\alpha\rightarrow0}	\frac{\gamma_{d,\alpha}}{\alpha}=\lim_{\alpha\longrightarrow0}\frac{\Gamma(\frac{d-\alpha}2)}{\pi^\frac{d}22^{\alpha+1}\frac{a}{2}\Gamma(\frac{\alpha}{2})}=\lim_{\alpha\rightarrow0}\frac{\Gamma(\frac{d-\alpha}2)}{\pi^\frac{d}22^{\alpha+1}\Gamma(\frac{\alpha}{2}+1)}=\frac{\Gamma(\frac{d}2)}{2\pi^\frac{d}2}=\frac{1}{\sigma_{d-1}},
		\end{equation*}
		the conclusions follows. 
	\end{proof}
	
	As a consequence, the Riesz kernel is an approximation of the identity, and it was observed by Kurokawa \cite{Kurokawa1981}, in the sense that
	\begin{equation*}
		I_\alpha * f \longrightarrow f, \quad \text{ as $\alpha\rightarrow 0$},
	\end{equation*}
	for instance, in $L^p(\R^d)$, if $f\in L^p(\R^d)\cap L^q(\R^d)$, $1<q<p$ or pointwise at each point $x$ of the Lebesgue set of $f\in L^p(\R^d)$, $1\leq p<\infty$. In particular, if $D u\in \bs L^p(\R^d)\cap \bs L^q(\R^d)$, with $1<q<p$, as observed in \cite{RodriguesSantos2019},   we have $D^su\underset{s\rightarrow1}{\longrightarrow}Du$ in $\bs L^p(\R^d)$. We shall need the following stronger result which is also a consequence of this observation (see also Proposition 2.10 of \cite{Kurokawa1981} for the pointwise convergence).
	
	\begin{theorem}\label{ialphag} If $g\in L^p(\R^d)\cap L^\infty(\R^d)\cap C(\R^d)$, for $p>1$, is uniformly continuous in $\R^d$, then
		\begin{equation*}
			\lim_{\alpha\rightarrow 0}\|I_\alpha * g-g\|_{L^\infty(\R^d)}=0.
		\end{equation*}
	\end{theorem}
	\begin{proof}
		Let $\eps>0$ and $0<\delta<1$ be such that
		\begin{equation*}
			|z-x|\leq \delta\ \Rightarrow\ |g(z)-g(x)|\leq \eps,\quad 	\forall x,z\in\R^d.
		\end{equation*}
		
		Using Lemma \ref{I_II_II} consider $\alpha_0$ such that, for $0<\alpha<\alpha_0$,
		\begin{equation*}
			\big|\|I_\alpha\|_{L^1(B(0,\delta))}-1\big|\leq \eps, \quad 
			\|I_\alpha\|_{L^{p'}(\R^d\setminus B(0,\delta))}\leq \eps.
		\end{equation*}
		Then, for all $x\in\R^d$,
		\begin{align*}
			(I_\alpha * g)(x)-g(x)=&
			\int_{B(0,\delta)}I_\alpha (y) g(x-y)\,dy+\int_{\R^d\setminus B(0,\delta)}I_\alpha (y) g(x-y)\,dy -g(x)\\
			=&
			\int_{B(0,\delta)}I_\alpha (y)( g(x-y)-g(x))\,dy+
			g(x)\Big(\int_{B(0,\delta)}I_\alpha (y) \,dy-1\Big)\\ &+
			\int_{\R^d\setminus B(0,\delta)}I_\alpha (y) g(x-y)\,dy.
		\end{align*}
		Hence
		\begin{align*}
			|(I_\alpha * g)(x)-g(x)|\leq&
			\eps \|I_\alpha\|_{L^1(B(0,\delta))}+\|g\|_{L^\infty(\R^d)}\big(\|I_\alpha\|_{L^1(B(0,\delta))}-1\big)\\
			&\hspace{3mm}+\|I_\alpha\|_{L^{p'}(\R^d\setminus B(0,\delta))}\|g\|_{L^p(\R^d)}\\
			\leq & \eps^2+\eps\big(\|g\|_{L^\infty(\R^d)}+\|g\|_{L^p(\R^d)}\big)
		\end{align*}
		and the conclusion follows.
	\end{proof}
	
	As it was proved in \cite[Theorem 1.2]{ShiehSpector2015}, the fractional gradient satisfies
	\begin{equation}\label{xx}
		D^su=I_{1-s}Du=I_{1-s}*Du, 
	\end{equation}
	at least for functions $u\in C^\infty_c(\R^d)$, although that proof is equally valid for functions only in $C^1_c(\R^d)$, see \cite[Proposition 2.2]{ComiStefani2019_2}. As a consequence of well-known properties of the Riesz potential, \eqref{xx} is then also valid for functions $u$ in the usual Sobolev space $W^{1,p}(\R^d)$, $1<p<\infty$, since $Du\in \bs L^p(\R^d)$. In particular, as an immediate consequence of Theorem \ref{ialphag}, we obtain the uniform approximation of continuous gradients by their fractional gradients.
	
	\begin{corollary}\label{2.3}
		For $w\in C^1_c(\R^d)$ we have 
		\begin{equation}\label{Dswdw}
			D^sw\underset{s\rightarrow1}{\longrightarrow}Dw\quad \text{in $\bs L^\infty(\R^d)$.}
		\end{equation}
	\end{corollary}
	
	\begin{remark}\label{remark2.4}
		The convergence in \eqref{Dswdw} has been shown with a different proof for functions in $C^2_c(\R^d)$ and, if $w\in W^{1,p}(\R^d)$, also in $\bs L^p(\R^d)$ for $1\leq p< \infty$, respectively in Proposition 4.4 and in Theorem 4.11 of \cite{ComiStefani2019}. This property can be seen as a localization of the fractional gradient. It has also been shown for functions in $W^{1,p}(\R^d)$ for $1<p<\infty$, in \cite[Theorem 3.2]{BellidoCuetoMora-Corral2021}.	
	\end{remark}
	
	For smooth functions with compact support, as it was observed in \cite{ComiStefani2019_2}, the distributional Riesz fractional gradient $D^s$ can also be defined for $0<s<1$ by 
	\begin{equation}
		\label{dsigma2}
		D^s u(x)=\mu_s\int_{\R^d}\frac{u(x)-u(y)}{|x-y|^{d+s}}\frac{x-y}{|x-y|}\,dy,
	\end{equation}
	where $\mu_s=(d+s-1)\gamma_{d,1-s}$ is bounded and $\lim_{s\rightarrow 1}\mu_s=0$.
	
	Let $\Omega$ be a bounded open subset of $\R^d$ and set 
	\begin{equation*}
		\Omega_R=\{x\in\R^d:d(x,\Omega)<R\}, \quad \text{for $R>0$.}
	\end{equation*}
	
	In this work, for a function $u$ defined in $\Omega$, we still denote  its extension by zero to $\R^d$ by $u$.
	
	From \eqref{xx} or \eqref{dsigma2}, we see that for a function $u\in C^1_c(\Omega)$, while $Du=0$ in $\R^d\setminus\Omega$, $D^su$ is in general different from zero in the whole $\R^d$. Nevertheless the following remark holds.
	
	\begin{remark}\label{dsigma_infty} For $u\in C^1_c(\Omega)$, from \eqref{dsigma2} we easily obtain 
		\begin{equation*}
			|D^s u(x)|\leq \frac{\mu_s }{d(x,\Omega)^{d+s}}\|u\|_{L^1(\Omega)}, \quad \forall x\in\R^d\setminus \overline{\Omega},
		\end{equation*}
		and, consequently, for all $R>0$ 
		\begin{equation*}
			\lim_{s\rightarrow 1}\|D^s u\|_{L^\infty(\R^d\setminus\Omega_R)}=0.
		\end{equation*}
	\end{remark}
	
	Now, for $u\in C^\infty_c(\R^d)$ and $1\leq p<\infty$, $0<s\leq 1$, we introduce the norms
	\begin{equation*}
		\|u\|_{\Lambda^{s,p}}=\Big(\|u\|_{L^p(\R^d)}^p+\|D^s u\|_{\bs L^p(\R^d)}^p\Big)^\frac1p
	\end{equation*}
	and we define the Banach spaces
	\begin{equation*}
		\Lambda^{s,p}(\R^d)=\overline{C^\infty_c(\R^d)}^{\|\,\cdot\,\|_{\Lambda^{s,p}}},
	\end{equation*}
	where we recognize $\Lambda^{1,p}(\R^d)=W^{1,p}(\R^d)$, as the usual Sobolev spaces.
	
	For $1<p<\infty$, in \cite{ShiehSpector2015} it was proved that $\Lambda^{s,p}(\R^d)$ (denoted there as $X^{s,p}(\R^d)$) is equal to $\{u\in L^p(\R^d): u=g_s*f, \text{ for some }f\in L^p(\R^d)\}$, where $g_s$ are the Bessel potentials, for $s\in\R$, which were introduced in $1960$ by A. Cald\'eron and J. L. Lions. They are also called Bessel potential spaces or generalised Sobolev spaces (see \cite[p. 219]{Adams1975} or \cite{Campos2021}). It is worth to recall that we have $\Lambda^{s+\eps,p}(\R^d)\hookrightarrow W^{s,p}(\R^d)\hookrightarrow \Lambda^{s-\eps,p}(\R^d)$, if $1<p<\infty$ and $s>\eps>0$, where $W^{s,p}(\R^d)$ denotes the fractional Sobolev-Slobodeckij spaces. In fact, $\Lambda^{k,p}(\R^d)=W^{k,p}(\R^d)$ for nonnegative integers $k$ or when $p=2$ and $s>0$, being  $\Lambda^{s,2}(\R^d)=W^{s,2}(\R^d)=H^{s}(\R^d)$ Hilbert spaces.
	
	Also in \cite{ShiehSpector2015} it was shown the fractional Sobolev inequality for $1<p<\infty$ and $0<s<1$,
	\begin{equation}\label{trudinger}
		\|u\|_{L^{p^*}(\R^d)}\leq C_*\|D^su\|_{\bs L^p(\R^d)},\quad \forall u\in C^\infty_c(\R^d), 
	\end{equation}
	for a constant $C_*>0$, where $p^*=\frac{dp}{d-sp}$, if $sp<d$, as well as the fractional Trudinger ($p^*<\infty$, if $sp=d$) and Morrey ($p^*=\infty$, if $sp>d$) inequalities.
	If $sp>d$, in the left side of \eqref{trudinger}, we may take  the semi-norm of $\beta$-H\"older continuous functions, $0<\beta=s-\frac{d}{p}$.
	
	For an open bounded set $\Omega\subset\R^d$, we define the subspace, for $0<s\leq 1$, 
	\begin{equation}\label{sp0}
		\Lambda^{s,p}_0(\Omega)=\overline{C^\infty_c(\Omega)}^{\,\|\,\cdot\,\|_{\Lambda^{s,p}}},\quad 1<p<\infty. 
	\end{equation}
	Clearly, considering the smooth functions with compact support trivially extended by zero outside their support, we have $\Lambda^{s,p}_0(\Omega)\subset \Lambda^{s,p}(\R^d)$.  We observe that, by definition, for $u\in \Lambda^{s,p}_0(\Omega)$ the $D^s u$ is the limit in $\bs L^p(\R^d)$ of $D^s u_n$, for some sequence $u_n\in C^\infty_c(\Omega)$.  Observing that, for $\varphi\in C^\infty_c(\Omega)$, we have
	$$\int_{\R^d}\varphi\,D^su_n=-\int_{\R^d}u_n\,D^s\varphi=-\int_{\R^d}u_n(I_{1-s}D\varphi)=-\int_{\R^d}(I_{1-s}u_n) D\varphi,$$
	by  using Fubini's Theorem, and letting $n\rightarrow\infty$, we conclude that
	$D^su=D(I_{1-s}u)$, i.e. $D^su$ is the distributional Riesz fractional gradient of $u$. Moreover, in the limit, we may also conclude that $u\in\Lambda^{s,p}_0(\Omega)$  also satisfies
	\begin{equation}\label{wDs}
		\int_{\R^d}\varphi D^su =-\int_{\R^d}uD^s\varphi,  \qquad\forall \varphi\in C^\infty_c(\Omega).
	\end{equation}

	From \eqref{trudinger} we obtain a Poincar\'e inequality 
	\begin{equation}\label{PoincareSobolev}
		\|u\|_{L^p(\Omega)}\le C_{p}\|D^s u\|_{\bs L^p(\R^d)},\quad\forall u\in \Lambda^{s,p}_0(\Omega),
	\end{equation}
	for some $C_p>0$, and in $\Lambda^{s,p}_0(\Omega)$ we shall use  the equivalent norm
	\begin{equation}\label{p-norm}
		\|u\|_{\Lambda^{s,p}_0(\Omega)}=\|D^su\|_{\bs L^p(\R^d)}.
	\end{equation}
	
	We can extend the definition \eqref{sp0} for $p=1$ and define
	\begin{equation*}
		\Lambda^{s,\infty}_0(\Omega)=\Big\{u\in \bigcap_{1< p<\infty}\Lambda^{s,p}_0(\Omega): D^su\in \bs L^\infty(\R^d)\Big\}.
	\end{equation*}
	The fractional Poincar\'e inequality \eqref{PoincareSobolev} can be made more precise with respect to $s$, $0<s<1$, to also include the limit cases $p=1$ and $p=\infty$.
	
	\begin{proposition}
		Let $\Omega\subset \R^d$ be a bounded open set. Then there exists a constant $C_0=C_0(\Omega,d)>0$ such that,
		for all $0<s<1$ and $1\leq p\leq \infty$,
		\begin{equation}\label{PoincareSobolevBellido}
			\|u\|_{L^p(\Omega)}\le \tfrac{C_0}{s}\|D^s u\|_{\bs L^p(\R^d)},\quad \forall u\in \Lambda^{s,p}_0(\Omega).
		\end{equation}
	\end{proposition}
	\begin{proof}
		For $1<p<\infty$, this is Theorem 2.9 of \cite{BellidoCuetoMora-Corral2021}, but the same proof is still valid for $p=1$. Since $C_0$ is independent of $p$, the case $p=\infty$ is obtained by letting $p\rightarrow \infty$ in \eqref{PoincareSobolevBellido}.
	\end{proof}
	
	In addition, in a bounded open set $\Omega\subset \R^d$ satisfying the extension property, it is well known that $\Lambda^{s,2}_0(\Omega)=W^{s,2}_0(\Omega)=H^s_0(\Omega)$ (see, for instance, \cite{LoRodrigues2023}).
	
	Although there is no monotone inclusions in $p$ of $L^p(\R^d)$ the following result holds.
	
	\begin{theorem}\label{OmegaR} Let $\Omega\subset\R^d$ be a bounded open set,  $p\in[1,\infty)$ and $0<s<1$. Then  there exists a positive constant $C=\big(1+\frac{1}{(p-1)d+ps}\big)C(d,\Omega)$, such that, for $R\geq 1$, 
		\begin{equation}\label{desigualdadeomegaR}
			\int_{\R^d\setminus \Omega_R}|D^s u(x)|^pdx\leq \frac{\mu_s^pC}{R^{(p-1)d+ps}} \|u\|_{L^1(\Omega)}^p, \quad \forall u\in\Lambda_0^{s,p}(\Omega).	
		\end{equation}
		
		As a consequence, the following inclusions hold
		\begin{equation}\label{Omegapq}
			\Lambda_0^{s,q}(\Omega)\subseteq \Lambda_0^{s,p}(\Omega),\quad 1\leq p< q<\infty,
		\end{equation}
		and are continuous, since there exists $C_{p,q}>0$ such that
		\begin{equation}\label{PoincareSobolevLambda}
			\|D^s u\|_{\bs L^p(\R^d)}\leq C_{p,q} \|D^s u\|_{\bs L^q(\R^d)},\quad u\in \Lambda_0^{s,q}(\Omega).
		\end{equation}
		
		In addition, $C_{1,q}=\frac{E}{s}$, where $E$ is independent of $s$, and $C_{p,q}$ is independent of $s$, if $p>1$.
	\end{theorem}
	\begin{proof} It is enough to consider $u\in C^\infty_c(\Omega)$. If $\delta(\Omega)$ denotes the diameter of $\Omega$, consider $S=\frac{1}{2}\delta(\Omega)+R$ and $z$ such that $\Omega_R\subseteq B(z,S)$. Consider the annulus $A_n=B(z,S+n+1)\setminus B(z,S+n)$, for each $n\in\N_0$.
		
		Letting $\omega_d=\big|B(0,1)\big|$ and $\mu_s$ be as in \eqref{dsigma2}, we have
		\begin{align*}
			\tfrac{1}{\mu_s^p}	\int_{\R^d\setminus B(z,S)}|D^s u(x)|^pdx
			\leq&\int_{\R^d\setminus B(z,S)}\bigg|\int_\Omega \tfrac{|u(y)|}{|x-y|^{d+s}}dy\bigg|^pdx\\
			=&
			\sum_{n=0}^\infty\int_{A_n}\bigg(\int_\Omega \tfrac{|u(y)|}{|x-y|^{d+s}}\,dy\bigg)^pdx\\
			\leq&
			\sum_{n=0}^\infty\int_{A_n}\bigg(\int_\Omega \tfrac{|u(y)|}{(n+R)^{d+s}}\,dy\bigg)^pdx\\
			=& 
			\sum_{n=0}^\infty\tfrac{\omega_d\big[(S+n+1)^d-(S+n)^d\big]}{(n+R)^{p(d+s)}} \|u\|_{L^1(\Omega)}^p.
		\end{align*}
		By the Lagrange theorem, there exists  $\nu\in(n,n+1)$ such that $(n+1+S)^d-(n+S)^d=d(\nu+S)^{d-1}\leq d(n+1+S)^{d-1}$
		and then, as $S\geq 1$ and $n+R\geq\frac{R}{S}(n+S)$,
		\begin{align*}
			\tfrac{1}{\mu_s^p}	\int_{\R^d\setminus B(z,S)} & |D^s u(x)|^p  dx
			\leq d\left(\tfrac{S}{R}\right)^{p(d+s)}	\omega_d\sum_{n=0}^\infty
			\big(\tfrac{n+1+S}{n+S}\big)^{d-1}\frac{1}{(n+S)^{(p-1)d+ps+1}} \|u\|_{L^1(\Omega)}^p\\
			&\leq d\left(\tfrac{S}{R}\right)^{p(d+s)}	\omega_d2^{d-1}\sum_{n=0}^\infty
			\frac{1}{(n+S)^{(p-1)d+ps+1}} 				 \|u\|_{L^1(\Omega)}^p\\ 
			&\leq d\left(\tfrac{S}{R}\right)^{p(d+s)}	\omega_d2^{d-1}\bigg[ \frac{1}{S^{(p-1)d+ps+1}}+\int_S^\infty \frac{1}{x^{(p-1)d+ps+1}}\,dx\bigg] \|u\|_{L^1(\Omega)}^p\\
			&= d\left(\tfrac{S}{R}\right)^{p(d+s)}	\omega_d2^{d-1}\bigg[\frac{1}{S^{(p-1)d+ps+1}}+\frac{1}{(p-1)d+ps}\frac{1}{S^{(p-1)d+ps}}\bigg] \|u\|_{L^1(\Omega)}^p\\
			&\leq d\left(\tfrac{S}{R}\right)^{p(d+s)}	\omega_d2^{d-1}\bigg(1+\frac{1}{(p-1)d+ps}\bigg)\frac{1}{S^{(p-1)d+ps}}\|u\|_{L^1(\Omega)}^p\\
			&= d\left(\tfrac{S}{R}\right)^d	\omega_d2^{d-1}\bigg(1+\frac{1}{(p-1)d+ps}\bigg)\frac{1}{R^{(p-1)d+ps}} \|u\|_{L^1(\Omega)}^p.
		\end{align*}
		
		On the other hand,
		\begin{align*}
			\tfrac{1}{\mu_s^p}	\int_{B(z,S)\setminus \Omega_R}|D^s u(x)|^pdx
			\leq &\int_{B(z,S)\setminus \Omega_R}\bigg(\int_\Omega \tfrac{|u(y)|}{|x-y|^{d+s}}\,dy\bigg)^pdx\\
			\leq&\int_{B(z,S)\setminus \Omega_R}\bigg(\int_\Omega \tfrac{|u(y)|}{R^{d+s}}\,dy\bigg)^pdx\\
			\leq&
			\omega_d\big(\tfrac{S}{R}\big)^d	\tfrac{1}{R^{(p-1)d+ps}} \|u\|_{L^1(\Omega)}^p.
		\end{align*}
		
		As $\frac{S}{R}\leq 1+\tfrac{1}{2}\delta(\Omega)$ we have
		\begin{equation*}
			\tfrac{1}{\mu_s^p}	\int_{\R^d\setminus \Omega_R}|D^s u(x)|^pdx\leq 
			\omega_d\big(1+\tfrac{1}{2}\delta(\Omega)\big)^d\left[d2^{d-1}\bigg(1+\frac{1}{(p-1)d+ps}\bigg)+1\right]	\frac{1}{R^{(p-1)d+ps}} \|u\|_{L^1(\Omega)}^p,
		\end{equation*}
		from where we obtain \eqref{desigualdadeomegaR}.

		For the inclusion \eqref{Omegapq}, by considering $R=1$,  there exists $C_1$ such that
		\begin{align*}
			\int_{\R^d}|D^s u(x)|^pdx	& = 	\int_{\R^d\setminus \Omega_1}|D^s u(x)|^pdx+	\int_{\Omega_1}|D^s u(x)|^pdx\\
			& \leq  C\mu_s^p\|u\|_{L^1(\Omega)}^p+|\Omega_1|^{\frac{1}{p}-\frac{1}{q}}\|D^s u\|_{\bs L^q(\Omega_1)}^p\\
			& \leq  C(\max_{s\in[0,1)}\mu_s^p)\,|\Omega|^{\frac{1}{q'}}\|u\|_{L^q(\Omega)}^p+|\Omega_1|^{\frac{1}{p}-\frac{1}{q}}\|D^s u\|_{\bs L^q(\Omega_1)}^p\\
			& \leq  	C_1	\|u\|_{\Lambda_0^{s,q}}^p
		\end{align*}
		by using Poincar\'e inequality \eqref{PoincareSobolevBellido}, yielding the conclusion.
	\end{proof}
	As in \eqref{p-norm} we define in $\Lambda^{s,\infty}_0(\Omega)$ the topology induced by $\| u\|_{\Lambda^{s,\infty}_0(\Omega)}=\|D^s u\|_{\bs L^\infty(\R^d)}$, which is a norm by Poincar\'e inequality. 
	
	\begin{proposition}\label{inftyp}
		There exists a constant $C_{p,\infty}>0$, which is independent of $s\in [\sigma,1)$ for each $\sigma>0$, such that 
		\eqref{PoincareSobolevLambda} holds for $q=\infty$.  In particular the inclusion 
		\begin{equation*}
			\Lambda^{s,\infty}_0(\Omega)\subset \Lambda^{s,p}_0(\Omega)    
		\end{equation*}  
		is continuous for all $p\geq 1$, and $\Lambda^{s,\infty}_0(\Omega)$ is a Banach space.
	\end{proposition}
	\begin{proof}	
		From Theorem \ref{OmegaR}, for $R\geq 1$  there exists $C>0$ independent of $R$, such that for all $u\in \Lambda_0^{s,\infty}(\Omega)$,
		\begin{align*}
			\int_{\R^d}|D^s u(x)|^pdx &\leq\int_{\Omega_R}|D^s u(x)|^pdx+\frac{C\mu_s^p}{R^{(p-1)d+ps}} \|u\|_{L^1(\Omega)}^p\\
			&\leq |\Omega_R|\|D^s u\|_{\bs L^\infty(\R^d)}^p+\frac{C\mu_s^p|\Omega|^{p-1}}{R^{(p-1)d+ps}} \|u\|_{L^p(\Omega)}^p\\
			&\leq |\Omega_R|\|D^s u\|_{\bs L^\infty(\R^d)}^p+\frac{C_0^p}{s^p}\frac{C\mu_s^p|\Omega|^{p-1}}{R^{(p-1)d+ps}} \|D^su\|_{L^p(\R^n)}^p
		\end{align*}
		by \eqref{PoincareSobolevBellido}.   
		Choosing 
		\begin{equation*}
			R=\max\left\{1,\left(\frac{2C_0^pC\mu_s^p|\Omega|^{p-1}}{s^p}\right)^{\frac{1}{(p-1)d+ps}}\right\}
		\end{equation*}
		we obtain 
		\begin{equation*}
			\|D^su\|_{L^p(\Omega)}\leq 2^{\frac{1}{p}} |\Omega_R|^{\frac{1}{p}} \|D^s u\|_{\bs L^\infty(\R^d)},
		\end{equation*} 
		which yields the continuity of the embedding $\Lambda^{s,\infty}_0(\Omega)\subset \Lambda^{s,p}_0(\Omega)$.	
		
		Finally,  since a Cauchy sequence $(u_n)$ in $\Lambda^{s,\infty}_0(\Omega)$ is also, for all $1<p<\infty$, a Cauchy sequence in the nested Banach spaces $\Lambda^{s,p}_0(\Omega)$, its common limit $u\in{\bigcap_{1<p<\infty}\Lambda^{s,p}_0(\Omega)}$. As $D^s u_n$ are uniformly bounded, then $D^s u$ is bounded, and therefore $u\in\Lambda^{s,\infty}_0(\Omega)$.
	\end{proof}
	
	\begin{remark}\label{bellido}
		The inclusion \eqref{Omegapq} was also obtained independently in \cite[Corollary 2.4.1]{Campos2021}, as a consequence of an interesting variant of the Poincar\'e inequality, see \cite[Theorem 2.4.3]{Campos2021}, for some constant $C_1=C_1(\Omega,\Omega_1,d)>0$,
		\begin{equation*}\|u\|_{L^p(\Omega)}\leq \tfrac{C_1}{s}\|D^su\|_{\bs L^p(\Omega_1)},\quad \forall u\in \Lambda_0^{s,p}(\Omega),
		\end{equation*}
		for an open set $\Omega_1\supset B(0,2R)\supset \Omega$, with $R>1$, $1<p<\infty$ and $0<s<1$.
	\end{remark}
	\begin{remark}\label{2.10} We note that, for  $p\in[1,\infty)$ we have the inclusions $ W^{1,p}_0(\Omega)\subset \Lambda^{s,p}_0(\Omega)\subset\Lambda^{\sigma,p}_0(\Omega)$, for $0<\sigma<s<1$. We may conclude that, as  a consequence of \eqref{desigualdadeomegaR},  for $R\geq 1$, as $\displaystyle\lim_{s\rightarrow 1}\mu_s=0$,
		\begin{equation*}
			\lim_{s\rightarrow 1}	\|D^s u\|_{L^p(\R^d\setminus\Omega_R)}=0, \quad\forall u\in W^{1,p}_0(\Omega).
		\end{equation*}
	\end{remark}
	
	\begin{remark}
		Also from Theorem \ref{OmegaR}, for $R\ge1$ we can take the limit as $p\rightarrow\infty$ in \eqref{desigualdadeomegaR}, to conclude (compare with Remark \ref{dsigma_infty}):
		\begin{equation*}
			\|D^su\|_{\bs L^\infty(\R^d\setminus\Omega_R)}\le\tfrac{\mu_s}{R^{d+s}}\|u\|_{L^1(\Omega)},\quad\forall u\in\Lambda^{s,\infty}_0(\Omega),\ 0<s<1.
		\end{equation*}
	\end{remark}
	
	\begin{remark}\label{r2.12}
		We denote by $W^{1,\infty}_0(\Omega)$  the space of Lipschitz functions vanishing on the boundary of $\Omega$. Extending a function $u\in W^{1,\infty}_0(\Omega)$ by zero outside $\Omega$ and using definition \eqref{dsigma2}, we have, for each $x\in\R^d$,
		\begin{align*}\left|D^s u(x)\right|&\le\mu_s\int_{\R^d}\frac{|u(x)-u(y)|}{|x-y|^{d+s}}\,dy\\
			&=\mu_s\|Du\|_{\bs L^\infty(\Omega)}\int_{\{|x-y|\le 1\}}\frac{dy}{|x-y|^{d+s-1}}+\mu_s\, 2\|u\|_{L^\infty(\Omega)}\int_{\{|x-y|>1\}}\frac{dy}{|x-y|^{d+s}}\\
			&\le\mu_s\,C_s\|Du\|_{\bs L^\infty(\Omega)},
		\end{align*}
		for a finite $C_s>0$, by the Poincar\'e inequality and since both integrals are finite for $s\in(0,1)$. Consequently,
		\begin{equation}\label{1sinfty}
			W^{1,\infty}_0(\Omega)\subset\Lambda^{s,\infty}_0(\Omega),\quad\forall s\in(0,1).
		\end{equation}
	\end{remark}
	
	\vspace{4mm}
	
	The Lions-Calder\'on spaces $\Lambda_0^{s,p}(\Omega)$, $0<s<1$, $1<p<\infty$, similarly to the Sobolev-Slobodeckij spaces $W^{s,p}_0(\Omega)$, have continuous and compact embeddings of Sobolev and Rellich-Kondrachov-type for $\Omega\subset\R^d$ open and bounded,
	\begin{equation}\label{2.19}
		\Lambda_0^{s,p}(\Omega)\subset L^q(\Omega), 
	\end{equation}
	for $q\in \big[1,\frac{dp}{d-sp}\big]$ if $sp<d$, for all $q\geq 1$ if $sp=d$, and for $q=\infty$ if $sp>d$,  the embeddings being compact in the case $sp<d$ only for $q<\frac{dp}{d-sp}=p^*$. Also the embeddings 
	\begin{equation}\label{Morrey}
		\Lambda_0^{s,p}(\Omega)\subset C^{0,\beta}(\overline\Omega),\quad\text{for $sp>d$} 
	\end{equation}
	are continuous for $0<\beta\leq s-\frac{d}{p}$ and compact for $0<\beta< s-\frac{d}{p}$, where $C^{0,\beta}(\overline\Omega)$ denotes the space of H\"older continuous functions in $\overline{\Omega}$ of exponent $\beta$. Consequently, by Proposition \ref{inftyp}, we have the compact embeddings
	\begin{equation}\label{2.21}
		\Lambda_0^{s,\infty}(\Omega)\subset C^{0,\beta}(\overline\Omega)\subset L^\infty(\Omega),\quad \text{for $0<\beta<s<1$}.
	\end{equation}
	
	We also have the non-trivial compact embeddings
	\begin{equation}\label{camposembedding}
		\Lambda_0^{s,p}(\Omega)\subset \Lambda_0^{\sigma,p}(\Omega), \quad 0<\sigma<s<1,\quad 1<p<\infty,
	\end{equation} 
	which proof can be found in \cite[pg. 65]{Campos2021} and is well known for $p=2$.
	
	We denote the dual space of $\Lambda_0^{s,p}(\Omega)$ by  $\Lambda^{-s,p'}(\Omega)$, $0<s<1$, $1<p<\infty$, and we have a similar characterization in terms of the fractional $s$-gradient as it was shown in \cite[Theorem 2.4.4, p. 66]{Campos2021} for bounded and unbounded open domains $\Omega\subset\R^d$.
	\begin{proposition}
		Let $0<s<1$, $1<p<\infty$ and $F\in \Lambda^{-s,p'}(\Omega)$. Then there exist functions $f_0\in L^{p'}(\Omega)$ and $f_1,\ldots,f_d\in L^{p'}(\R^d)$ such that 
		\begin{equation}\label{dual}
			[F,v]_{s,p}=\int_\Omega f_0v+\sum_{j=1}^d \int_{\R^d}f_jD^s_jv,\quad \forall v\in \Lambda_0^{s,p}(\Omega).
		\end{equation}
	\end{proposition}
	
	When $p=\infty$  and  $f_0\in L^1(\Omega)$ and $f_1,\ldots,f_d\in L^1(\R^d)$, \eqref{dual} defines a linear form  in $\Lambda_0^{s,\infty}(\Omega)$. However, these forms do not exhaust $\Lambda_0^{s,\infty}(\Omega)'$.
	
	We shall also work with the dual of $L^\infty(\R^d)$,  which is also denoted as $ba(\R^d)$ (see, for instance, \cite{RaoRao1983} and \cite{AliprantisBorder2006}) and their elements are sometimes called charges. We recall (see Example 5, Section 9, Chapter IV of \cite{Yosida1980}) that an element $\lambda\in L^\infty(\R^d)'$ can be represented by a Radon integral
	\begin{equation}\label{Radon}
		\langle\lambda,\varphi\rangle=\int_{\R^d}\varphi d\lambda^*,\quad\forall \varphi\in L^\infty(\R^d),
	\end{equation}
	for a finitely additive measure $\lambda^*$, which is of bounded variation and absolutely continuous with respect to the Lebesgue measure in $\R^d$.
	
	We say that a charge $\lambda$ is positive, or simply $\lambda\geq 0$, if  $\langle\lambda,\varphi\rangle\ge0$ for any $\varphi\in L^\infty(\Omega)$, $\varphi\ge0$.
	
	Exactly as for the Lebesgue integral, we have the H\"older inequality for positive charges (see \cite[p.122]{RaoRao1983}).
	
	\begin{proposition} \label{desigHolder} 
		Let $p>1$ and $\lambda\in L^\infty(\R^d)'$ be positive. Then
		\begin{equation*}
			\big|\langle\lambda,\varphi\psi\rangle\big|\le\langle\lambda,|\varphi|^p\rangle^\frac1{p}\langle\lambda,|\psi|^{p'}\rangle^\frac1{p'},\quad\forall \varphi,\psi\in L^\infty(\Omega).
		\end{equation*}
	\end{proposition}
	
	\section{Variational inequalities with $s$-gradient constraints} \label{section3}
	
	Let $\Omega\subset\R^d$ be open and bounded, with the extension property, i.e., the extension of $u\in H^s_0(\Omega)$ by zero in $\R^d\setminus\Omega$ is in $H^s(\R^d)$, $0<s\le1$. This holds, in particular, for domains with Lipschitz boundaries (see for instance \cite[Section 5]{DiNezaPalatucciValdinoci2012}). To consider $s$-gradient constrained problems, we define the following closed convex sets
	\begin{equation}\label{kg}
		\K^s_g=\big\{v\in H^s_0(\Omega):|D^sv|\le g\text{ a.e. in }\R^d\big\},\quad0<s\le1,
	\end{equation}
	for prescribed thresholds satisfying
	\begin{equation}\label{g0}
		g\in L^2_{loc}(\R^d),\quad g\ge0\text{ a.e. in }\R^d,
	\end{equation}
	or, in the bounded case,
	\begin{equation}\label{g*}
		g\in L^\infty_{loc}(\R^d),\quad g\ge0\text{ a.e. in }\R^d.
	\end{equation}
	In Lemma \ref{klimitado}, we will see that these assumptions on $g$ are enough for $\K_g^s$ to be bounded in $H^s_0(\Omega)$ and $\Lambda^{s,\infty}_0(\Omega)$, respectively.
	
	For $0<s\leq 1$, we define a bilinear form by letting
	\begin{equation}\label{bilinear}
		\mathscr L^s(u,v)=\int_{\R^d}AD^su\cdot D^sv+\int_\Omega \bs d u\cdot D^sv+\int_\Omega(\bs b\cdot D^su+cu)v.
	\end{equation}
	Here the principal part may be degenerate, under the assumption on the matrix $A=A(x)$:
	\begin{equation}\label{matrix}
		A(x)\bs\xi\cdot\bs\xi\ge0,\quad\forall\bs\xi\in\R^d,\ \text{a.e. }x\in\R^d.
	\end{equation}
	In addition, we shall assume that the coefficients of the bilinear form satisfy, when $u,v\in H^s_0(\Omega)$, 
	\begin{equation}\label{assumptions1}
		A\in L^\infty(\R^d)^{d^2},\bs b,\bs d\in\bs L^r(\Omega)\ \text{ and }\ c\in L^{\tfrac{r}2}(\Omega),\quad r>\tfrac{d}{s}
	\end{equation}
	or, in the bounded case, when $u,v\in\Lambda^{s,\infty}_0(\Omega)$,
	\begin{equation}\label{assumptions2}
		A\in L^{p_1}(\R^d)^{d^2}, \ {p_1}\in[1,\infty),\ \bs b,\bs d\in\bs L^1(\Omega)\ \text{ and }\ c\in L^{1}(\Omega).
	\end{equation}
	
	Similarly, for $0<s\le1$, we may define the linear form
	\begin{equation}\label{fs}
		F_s(v)=[F,v]_s=\int_\Omega f_{\#}v+\int_{\R^d}\bs f\cdot D^sv
	\end{equation}
	for any $v\in H^s_0(\Omega)$,
	with
	\begin{equation}\label{assumptions_s1}
		f_\#\in L^{2^{\#}}(\Omega),\quad\bs f\in\bs L^2(\R^d),
	\end{equation}
	where, by the Sobolev embedding \eqref{trudinger}, $2^{\#}=\frac{2d}{d+2s}$ if $0<s<\tfrac{d}2$, or $2^{\#}=q$ for any $q>1$ when $s=\frac12$, and $2^{\#}=1$ when $\frac12<s\le1$ or, in the bounded case, for any $v\in\Lambda^{s,\infty}_0(\Omega)$, with
	\begin{equation}\label{assumptions_s2}
		f_{\#}\in L^1(\Omega),\quad\bs f\in\bs L^{q_1}(\R^d), \ {q_1}\in[1,\infty).
	\end{equation}
	Notice that in the case $s=1$, since $u,v$ and $Du,Dv$ are zero in $\R^d\setminus\Omega$, all the integration domains in \eqref{bilinear} and \eqref{fs} reduce to $\Omega$.
	
	\begin{theorem}\label{teo3.1} Assume \eqref{matrix}, and suppose that
		\begin{enumerate}
			\item[i)] either assumptions \eqref{g0}, \eqref{assumptions1}, and \eqref{assumptions_s1} hold, 
			\item[ii)] or assumptions \eqref{g*}, \eqref{assumptions2}, and \eqref{assumptions_s2} hold.
		\end{enumerate}
		
		Then, for $0<s\leq 1$, there exists a solution of the $s$-gradient constraint variational inequality
		\begin{equation}\label{iv}
			u\in\K^s_g:\qquad\mathscr L^s(u,v-u)\ge[F,v-u]_s,\quad\forall v\in\K_g^s.
		\end{equation}
	\end{theorem}
	
	We will use the following lemma in the proof of this theorem.
	
	\begin{lemma}\label{klimitado} For $1\leq p\leq \infty$ and $g\in L_{loc}^p(\R^d)$, with $g\geq 0$, the set $\K^s_g$ is bounded in $\Lambda^{s,p}_0(\Omega)$. More precisely, there exists $R=R(p,s)$ such that, for $u\in\K^s_g$,
		\begin{equation}\label{1sobrep}
			\|D^s u\|_{\bs L^p(\R^d)}\le 2^\frac1p \|g\|_{L^p(\Omega_R)},\ \text{if $p<\infty$}, \quad	\|D^s u\|_{\bs L^\infty(\R^d)}\leq \|g\|_{L^\infty(\Omega_R)}.
		\end{equation}
	\end{lemma}
	\begin{proof}
		By the Theorem \ref{OmegaR}, when $p<\infty$, choosing $R$ such that $\frac{C\mu_s^p}{R^{(p-1)d+ps}}\frac{C_0^p}{s^p}|\Omega|^{p-1}\leq \frac{1}{2}$, and using \eqref{desigualdadeomegaR}, we have
		\begin{align*}
			\|D^s u\|_{\bs L^p(\R^d)}^p&=
			\|D^s u\|_{\bs L^p(\R^d\setminus\Omega_{R})}^p+	\|D^s u\|_{\bs L^p(\Omega_{R})}^p\ \leq \ \frac{C\mu_s^p}{R^{(p-1)d+ps}}\|u\|_{L^1(\Omega)}^p+\int_{\Omega_{R}}|D^s u|^p\\
			&\le\tfrac12	\|D^s u\|_{\bs L^p(\R^d)}^p+\int_{\Omega_R}g^p,
		\end{align*}
		from where we obtain the first inequality. Letting $p\rightarrow\infty$, the second inequality follows.
	\end{proof}
	
	\begin{proof} (of Theorem \ref{teo3.1})
		This existence result in the  Hilbertian case i) is a consequence of a theorem of H. Br\'ezis (see \cite{{Brezis1968}} or \cite[Theorem 8.1, p. 245]{Lions1969}), since $\K^s_g$ is a nonempty, closed and bounded  convex set of $H^s_0(\Omega)$ and the operator $P:H^s_0(\Omega)\longrightarrow H^{-s}(\Omega)$ defined by
		\begin{equation}\label{p}
			[Pu,v]_s=\mathscr L^s(u,v)-[F,v]_s,\quad u,v\in\K^s_g
		\end{equation}
		is pseudo-monotone in $\K^s_g$, i.e., if $u_n\underset{n}\lraup u$ in $H^s_0(\Omega)$, for $u_n,u\in\K^s_g$ with $\displaystyle\varlimsup_n[Pu_n,u_n-u]_s\le0$ then
		\begin{equation}\label{pseudo}
			\varliminf_n[Pu_n,u_n-v]_s\ge[Pu,u-v]_s,\quad\forall v\in\K^s_g.
		\end{equation}
		
		Indeed, taking $u_n\underset{n}\lraup u$ in $H^s_0(\Omega)$, which by compactness of the embedding \eqref{2.19} (for $p=2$ and $1\le q<2^*=\frac{2d}{d-2s}$ if $2s<d$, for all $q\ge1$ if $2s=d$ and for $q=\infty$ if $2s>d$), we may assume also that $u_n\underset{n}\rightarrow u$ in $L^q(\Omega)$. Write $P$ in the form
		\begin{equation*}
			[Pu,w]_s=\int_{\R^d}AD^su\cdot D^sw+[Bu,w]_s
		\end{equation*}
		with
		\begin{equation*}
			[Bu,w]_s=\int_{\R^d}(\bs du-\bs f)\cdot D^sw+\int_\Omega(\bs b\cdot D^su+cu-f_{\#})w.
		\end{equation*}
		It is then clear that the assumptions \eqref{assumptions2} and \eqref{assumptions_s2} imply
		\begin{equation*}
			[Bu_n,u_n-v]_s\underset{n}{\rightarrow}[Bu,u-v]_s,
		\end{equation*}
		as $D^su_n\underset{n}\lraup D^su$ in $\bs L^2(\R^d)$, and we have 
		\begin{equation*}
			\int_\Omega(\bs b+\bs d)u_n\cdot D^su_n\underset{n}\longrightarrow\int_\Omega(\bs b+\bs d)u\cdot D^su\quad \text{ and }\  \int_\Omega cu_n^2\underset{n}\longrightarrow \int_\Omega cu^2.
		\end{equation*}
		
		Hence \eqref{pseudo} follows easily by noting that the assumption \eqref{matrix} implies $\displaystyle\int_{\R^d}A D^s(u_n-u)\cdot D^s(u_n-u)\ge 0$, and hence it suffices to take the limit inferior in
		\begin{equation}\label{liminf}
			\int_{\R^d}AD^su_n\cdot D^s u_n\ge\int_{\R^d}AD^su_n\cdot D^su+\int_{\R^d}AD^su\cdot D^su_n-\int_{\R^d}AD^su\cdot D^su.
		\end{equation}
		
		In the non-Hilbertian case, we start by approximating the data, in the respective spaces,  by smooth functions  with compact support $A_m$, $\bs b_m$, $\bs d_m$, $c_m$, $f_{\# m}$ and $\bs f_m$ and we let $u_m$ be a solution of the variational inequality \eqref{iv} with these data, which exists by the previous case. 
		
		As $(D^su_m)_m$ is bounded in $\bs L^ \infty(\R^d)$ by Lemma \ref{klimitado}, using the compact embedding \eqref{Morrey}, there exist a $u\in\Lambda^{s,\infty}_0(\Omega)$ and a $G\in\bs L^ \infty(\R^d)$, such that, for some subsequence, $u_m\underset m\longrightarrow u$ strongly in $L^\infty(\Omega)$  and $D^su_m\underset m\lraup G$ in $\bs L^\infty(\R^d)\text{-weak}^*$.  Using the distributional nature of $D^s$, we easily see that $G=D^su$. 	Thus, as $\Lambda^{s,\infty}_0(\Omega)$ is continuously included in $\Lambda^{s,p}_0(\Omega)$, we have $D^su_m\underset{m}{\lraup} D^su$ in $\bs L^{p}(\R^d)\text{-weak}$, for any $p<\infty$, in particular for $p=p_1'$ and $p=q_1'$.
		
		Using the above convergences, we immediately have, for any $v\in\K^s_g$,
		\begin{multline*}
			\int_{\R^d}(\bs d_m u_m-\bs f_m)\cdot D^s(v-u_m)+\int_\Omega(\bs b_m\cdot D^su_m+ cu_m-f_{\#}m)(v-u_m)\\
			\underset m\longrightarrow\int_{\R^d}(\bs du-\bs f)\cdot D^s(v-u_m)+\int_\Omega(\bs b\cdot D^su+cu-f_{\#})(v-u_m).
		\end{multline*}
		On the other hand, using the monotonicity of $A_m$, for any $v\in\K^s_g$ we have
		\begin{equation*}
			\int_{\R^d}A_m D^su_m\cdot D^s(v-u_m)\le \int_{\R^d}A_m D^sv\cdot D^s(v-u_m).
		\end{equation*}
		Since
		$\displaystyle\int_{\R^d}A_m D^sv\cdot D^s(v-u_m) \underset m\longrightarrow\int_{\R^d}A D^sv\cdot D^s(v-u)$ then we get that, for any $v\in\K_g^s$,
		\begin{equation*}
			\int_{\R^d}A D^sv\cdot D^s(v-u)+	\int_{\Omega}\bs d u\cdot D^s(v-u)+\int_\Omega(\bs b\cdot D^su+cu)(v-u) \geq [ F,v-u].
		\end{equation*}
		Choosing $v=u+t(w-u)\in\K^s_g$,  for $t\in(0,1)$, as test function,  we obtain
		\begin{equation*}
			\mathscr L^s(u,w-u)\ge[F,w-u]_s,\quad\forall w\in\K_g^s,
		\end{equation*}
		after letting $t\longrightarrow 0^+$. Therefore $u$ solves \eqref{iv}.
	\end{proof}
	
	\begin{remark} To obtain the uniqueness to \eqref{iv} it suffices to require the strict positivity of the bilinear form
		\begin{equation*}
			\mathscr L^s(u-\widehat u,u-\widehat u)>0,\quad\forall u,\widehat u\in\K_g^s: u\neq\widehat u,
		\end{equation*}
		which needs stronger assumptions on its coefficients.
	\end{remark}
	
	\begin{remark}\label{rem3.4}
		The constrained problem \eqref{iv} for $u\in\K_g^s$ determines the existence of an element $\Gamma=\Gamma(u)\in H^{-s}(\Omega)$ belonging to the sub-differential of the indicatrix function $I_{\K_g^s}$ of the convex set $\K_g^s$, i.e. $I_{\K_g^s}(v)=0$ if $v\in\K_g^s$, $I_{\K_g^s}(v)=+\infty$ if $v\in H^s_0(\Omega)\setminus\K_g^s$, (see \cite[p.203]{Lions1969}), which is given by
		\begin{equation*}
			\Gamma\equiv F-\mathscr L^su\in\partial I_{\K_g^s}\quad\text{ in } H^s_0(\Omega),
		\end{equation*}
		where $\mathscr L^s:\K_g^s\longrightarrow H^{-s}(\Omega)$ is the linear operator defined by the bilinear form as in \eqref{bilinear}. A main question is to relate $\Gamma$ to the solution $u$, for instance trough the existence of a Lagrange multiplier $\lambda$ such that $\Gamma=\lambda D^su$. This has been shown only in very special cases with the classical gradient ($s=1$) (see \cite{Brezis72} and \cite{RodriguesSantos2019_0}, for more references).
	\end{remark}
	
	The existence result of Theorem \ref{teo3.1} includes the degenerate case $A\equiv0$ in \eqref{iv}. On the other, when the matrix $A$ is strictly elliptic, i.e., if we replace \eqref{matrix} by assuming the existence of $a_*>0$, such that
	\begin{equation}\label{Aelliptic}
		A(x)\bs\xi\cdot\bs\xi\ge a_*|\bs\xi|^2,\quad\forall\bs\xi\in\R^d,\ \text{ a.e. }x\in\R^d,
	\end{equation}
	we may give the following sufficient condition for the bilinear form \eqref{bilinear} to be strictly coercive, by imposing
	\begin{equation}\label{coercive}
		\delta\equiv a_*-C_*\Big(\|\bs b+\bs d\|_{\bs L^\frac{d}s(\Omega)}+C_*\|c^-\|_{L^\frac{d}{2s}(\Omega)}\Big)>0.
	\end{equation}
	Here $C_*$ is the Sobolev constant of the embedding $H^s_0(\Omega)\hookrightarrow L^{2^*}(\Omega)$ ($2^*=\frac{2d}{d-2s}$, $2s<d$) and $c^-=\max\{0,-c\}$.
	
	\begin{theorem}\label{teo3.5}
		Let $A\in L^\infty(\R^d)^{d^2}$, $\bs b,\bs d\in\bs L^{\frac{d}s}(\Omega)$, $c\in L^\frac{d}{2s}(\Omega)$ satisfy \eqref{Aelliptic} and \eqref{coercive} for $2s<d$, and $g\in L^2_{loc}(\Omega)$. Then, for any $f_{\#}$ and $\bs f$ satisfying \eqref{assumptions_s1}, there exists a unique solution to \eqref{iv}. If $\widehat u$ denotes the solution to \eqref{iv} for $\widehat f_{\#}$ and $\widehat{\bs f}$, we have
		\begin{equation}\label{estimate}
			\|u-\widehat u\|_{H^s_0(\Omega)}\le\tfrac{C_*}\delta\|f_{\#}-\widehat f_{\#}\|_{L^{2^\#}(\Omega)}+\tfrac1\delta\|\bs f-\widehat{\bs f}\|_{\bs L^2(\R^d)}.
		\end{equation}	
	\end{theorem}
	\begin{proof}
		Using H\"older and Sobolev inequalities, we have that, for $v\in H^s_0(\Omega)$,
		\begin{equation*}
			\Big|\int_\Omega(\bs b+\bs d)vD^sv\Big|
			\le\|\bs b+\bs d\|_{\bs L^{\frac{d}{s}}(\Omega)}\|v\|_{L^{2^*}(\Omega)}
			\|D^sv\|_{\bs L^2(\R^d)}
			\le C_*\|\bs b+\bs d\|_{\bs L^\frac{d}s(\Omega)}\|D^sv\|^2_{\bs L^2(\R^d)},
		\end{equation*}
		and
		\begin{equation*}
			-\int_\Omega cv^2\le\int_\Omega c^-v^2\le C_*^2\|c^-\|_{L^\frac{d}{2s}(\Omega)}\|D^sv\|^2_{\bs L^2(\R^d)}.
		\end{equation*}
		
		Therefore, using \eqref{Aelliptic} and \eqref{coercive}, we obtain
		\begin{equation}\label{3.21}
			\begin{aligned}
				\mathscr L^s(v,v)&=\int_{\R^d}AD^sv\cdot D^sv+\int_\Omega(\bs b+\bs d)v\cdot D^sv+\int_\Omega cv^2\\
				&\ge a_*\int_{\R^d}|D^sv|^2-\Big(C_*\|\bs b+\bs d\|_{\bs L^{\frac{d}{s}}(\Omega)}+C_*^2\|c^-\|_{L^\frac{d}{2s}(\Omega)}\Big)\|D^sv\|^2_{\bs L^2(\R^d)}\\
				&=\delta\|D^sv\|^2_{\bs L^2(\R^d)}=\delta\|v\|^2_{H^s_0(\Omega)}.
			\end{aligned}
		\end{equation}
		As $L^{2^*}(\Omega)=L^{2^{\#}}(\Omega)'$, \eqref{assumptions_s1} implies that $F\in H^{-s}(\Omega)$, with $\|F\|_{H^{-s}(\Omega)}\le C_*\|f_{\#}\|_{L^{2^{\#}}(\Omega)}+\|\bs f\|_{\bs L^2(\R^d)}$ and Stampacchia's theorem immediately yields the existence and uniqueness of the solution to \eqref{iv}.
		
		Taking $v=\widehat u$ in \eqref{iv} for $u$ and $v=u$ in \eqref{iv} for $\widehat u$, and using \eqref{3.21}, we obtain
		\begin{equation*}
			\delta\|u-\widehat u\|_{H^s_0(\Omega)}^2\le\mathscr L^s(u-\widehat u,u-\widehat u)\le[F-\widehat F,u-\widehat u]_s\le\|F-\widehat F\|_{H^{-s}(\Omega)}\|u-\widehat u\|_{H^s_0(\Omega)}
		\end{equation*}
		and \eqref{estimate}  easily follows.
	\end{proof}
	
	\begin{remark}
		We observe that the assumption \eqref{assumptions1} is slightly stronger than the integrability conditions in Theorem \ref{teo3.5} for the case $2s<d$, including $s=1$. However, for $s\ge\frac{d}2$, we may have the assumption \eqref{assumptions1} with any $r>1$, when $s=1$, $d=2$, and even $\bs b$, $\bs d$, $c\in L^1$ when $s\ge\frac12$, $d=1$, with the respective norms in the assumption \eqref{coercive}.
		Note that Theorem \ref{teo3.5} extends Theorem 2.1 of \cite{AzevedoRodriguesSantos2022}, in which the coefficients $\bs b$, $\bs d$ and $c$ are zero.
	\end{remark} 
	
	The coercivity assumption \eqref{coercive} in the case of a bounded threshold of the $s$-gradient, under the stronger assumptions
	\begin{equation}\label{g_*}
		0<g_*\le g(x)\le g^*\quad\text{ for a.e. }x\in\R^d,
	\end{equation}
	also yields strong continuous dependence of the solutions of \eqref{iv} with respect to the variation of the coefficients of $\mathscr L^s$, of the data and of the threshold $g$.
	In fact, the assumption \eqref{g_*} can be weakened as follows
	\begin{equation}\label{gweak}
		g\in L^\infty_{loc}(\R^d)\text{ with positive lower bound in any compact and }\lim_{|x|\rightarrow\infty}g(x)|x|^{d+s}=\infty,
		\tag*{(\theequation)$_{loc}$}
	\end{equation}
	as it will be shown in Proposition \ref{locallocal}.
	
	\begin{theorem}\label{teo3.7}
		Let $u_i$ denote the solution of \eqref{iv} corresponding to the data $A_i$, $\bs b_i$, $\bs d_i$, $c_i$, $f_{\# i}$, $\bs f_i$, $g_i$, for $i=1,2$, satisfying \eqref{Aelliptic}, \eqref{coercive}, \eqref{assumptions2}, \eqref{assumptions_s2} and \eqref{g_*}. Then the following estimate holds with $p>\frac{d}s$ and $0<\gamma=s-\frac{d}p<s\le 1$,
		\begin{align}\label{dep_cont}
			\begin{aligned}
				\|u_1-u_2\|^p_{C^{0,\gamma}(\overline\Omega)}&+\|u_1-u_2\|^2_{H^s_0(\Omega)}\le C_1\|g_1-g_2\|_{L^\infty(\R^d)}\\
				&	+C_1'\Big(\|A_1-A_2\|_{L^{p_1}(\R^d)^{d^2}}+\|\bs b_1-\bs b_2\|_{\bs L^1(\Omega)}+\|\bs d_1-\bs d_2\|_{\bs L^1(\Omega)}\\
				&	+\|c_1-c_2\|_{L^1(\Omega)}+\|f_{\# 1}-f_{\# 2}\|_{L^1(\Omega)}+\|\bs f_1-\bs f_2\|_{\bs L^{q_1}(\R^d)}\Big),
			\end{aligned}	
		\end{align}
		where the positive constants $C_1$ and $C_1'$ depend on $\delta$, $g_*$, $g^*$, $d$, $s$, $\Omega$ and linearly on the $L^1$-norms of $A_i$, $\bs b_i$, $\bs d_i$, $c_i$, $f_{\# i}$, $\bs f_i$.
	\end{theorem}
	\begin{proof}
		Since $u_i\in\K^s_{g_i}\subset\Lambda^{s,\infty}_0(\Omega)$ and $g_i$ satisfies \eqref{g_*} for each $i=1,2$, we have
		\begin{equation*}
			\tfrac{s}{C_0}\|u_i\|_{L^\infty(\Omega)}\le\|D^su_i\|_{\bs L^\infty(\R^d)}\le g^*,
		\end{equation*}
		where $C_0>0$ is the Poincar\'e constant in \eqref{PoincareSobolevBellido}.
		
		Set $\eta=\|g_1-g_2\|_{L^\infty(\R^d)}$ and $\mu=\frac{g_*}{g_*+\eta}$. Observe that $u_{i_j}=\mu u_j\in\K^s_{g_i}$ ($i\neq j$, $i,j=1,2$) and so it can be used as test function in \eqref{iv} for $\mathscr L^s_i$ and $f_{\# i}$, $\bs f_i$. For $i=1,2$, we obtain
		\begin{equation*}
			\mathscr L_i^s(u_i,u_{i_j}-u_i)\ge[F_i,u_{i_j}-u_i]_s,
		\end{equation*}
		or equivalently,
		\begin{equation}\label{li}
			\mathscr L^s_i(u_i,u_i-u_j)\le[F_i,u_i-u_j]_s+\mathscr L^s_i(u_i,u_{i_j}-u_i)+[F_i,u_j-u_{i_j}]_s.
		\end{equation}
		
		Since $u_{i_j}-u_j=(\mu-1)u_j$ and $0\le 1-\mu\le\frac\eta{g_*}$, setting $M=\max\{g^*,\frac{C_0}s g^*\}$, we may estimate the middle term of \eqref{li} by
		\begin{align*}
			\Big|\mathscr L_i^s(u_i,u_{i_j}-u_j)\Big|&=\Big|(\mu-1)\mathscr L_i^s(u_i,u_j)\Big|\le\tfrac\eta{g_*}\Big|\mathscr L_i^s(u_i,u_j)\Big|\\
			&\le\tfrac{\eta}{g_*}M^2\Big(\|A_i\|_{L^1(\R^d)^{d^2}}+\|\bs b_i\|_{\bs L^1(\Omega)}+\|\bs d_i\|_{\bs L^1(\Omega)}+\|c_i\|_{L^1(\Omega)}\Big)=\eta\kappa_i
		\end{align*}
		and the last one by
		\begin{equation*}
			\big|[F_i,u_j-u_{i_j}]_s\big|=(1-\mu)\big|[F_i,u_j]_s\big|\le\tfrac{\eta}{g_*}M\big(\|f_{\# i}\|_{L^1(\Omega)}+\|\bs f_i\|_{\bs L^1(\R^d)}\big)=\eta \nu_i.
		\end{equation*}
		
		Setting $w=u_1-u_2$ and
		\begin{multline*}
			E_{21}=\int_{\R^d}(A_2-A_1)D^su_2\cdot D^sw+\int_\Omega(\bs d_2-\bs d_1)u_2\cdot D^sw\\
			+\int_\Omega\big((\bs b_2-\bs b_1)\cdot D^su_2 +(c_2-c_1)u_2\big)w,
		\end{multline*}
		by using \eqref{li} for $i=2$, we have
		\begin{equation}\label{3.25}
			-\mathscr L_1^s(u_2,w)=\mathscr L_2^s(u_2,-w)+E_{21}\le[F_2,-w]_s+\eta(\kappa_2+\nu_2)+e_{21},
		\end{equation}
		where
		\begin{multline*}
			\big|E_{21}\big|
			\le 2M^2\Big(\|A_1-A_2\|_{L^1(\R^d)^{d^2}}+\|\bs d_1-\bs d_2\|_{\bs L^1(\Omega)}\\
			+\|\bs b_1-\bs b_2\|_{\bs L^1(\Omega)}+\|c_1-c_2\|_{L^1(\Omega)}\Big) \equiv e_{21}.
		\end{multline*}
		
		Summing \eqref{li} for $i=1$ with \eqref{3.25} and using the coercivity \eqref{3.21}, we obtain
		\begin{align*}
			\delta\|w\|^2_{H^s_0(\Omega)}&\le\mathscr L_1^s(w,w)=\mathscr L^s_1(u_1,w)-\mathscr L_1^s(u_2,w)\\
			&\le[F_1,w]_s+\eta(\kappa_1+\nu_1)+[F_2,-w]_s+\eta(\kappa_2+\nu_2)+e_{21}\\
			&\le\eta(\kappa_1+\nu_1+\kappa_2+\nu_2)+e_{21}+\varphi_{12},
		\end{align*}
		where $\big|[F_1-F_2,w]_s\big|\le\varphi_{12}$, with $\varphi_{12}=2M\big(\|f_{\# 1}-f_{\# 2}\|_{L^1(\Omega)}+\|\bs f_1-\bs f_2\|_{\bs L^1(\R^d)}\big)$.
		
		To conclude \eqref{dep_cont} it suffices to use the continuous embedding \eqref{Morrey}, which guarantees the existence of a constant $C_\beta>0$ and $0<\beta=s-\frac{d}{p}<s\le1$, with $p>\max\{2,\frac{d}s\}$ in
		\begin{equation*}
			\big(\tfrac1{C_\beta}\big)^p\|w\|^p_{C^{0,\beta}(\overline\Omega)}\le\int_{\R^d}|D^sw|^p\le \|D^sw\|_{\bs L^\infty(\R^d)}^{p-2}\int_{\R^d}|D^sw|^2\le \big(g^*\big)^{p-2}\|w\|^2_{H^s_0(\Omega)}.
		\end{equation*}
	\end{proof}
	
	The following proposition shows that we can replace the assumption \eqref{g_*} by \ref{gweak} in the above theorem.
	
	\begin{proposition}\label{locallocal}
		Let $g\in L_{loc}^\infty(\R^d)$ be positively lower bounded in any compact set and such that 
		\begin{equation}\label{inftyinfty}
			\lim_{|x|\rightarrow+\infty}g(x)|x|^{d+s}=\infty.
		\end{equation}
		
		Then there exists $h\in L^\infty(\R^d)$, with positive lower bound in $\R^d$ and such that $\K_h^s=\K_g^s$. More precisely, we can choose $h=g\,\chi_{\Omega_R}+k\chi_{\R^d\setminus\Omega_R}$ for a certain $R>0$ and $k\ge \|g\|_{L^\infty(\Omega_R)}$.
	\end{proposition}
	\begin{proof}
		Using remarks \ref{dsigma_infty}, \ref{bellido} and the inequality \eqref{1sobrep} for $p=1$, there exists
		$R_0=R_0(s)>0$, independent of $g$, such that, for $R\geq R_0$, and $u\in\K_g^s$, we have
		\begin{equation*}
			|D^s u(x)|\leq \tfrac{C_0}{s}\frac{2\mu_s |\Omega_R|}{d(x,\Omega)^{d+s}}\|g\|_{L^\infty(\Omega_R)}, \qquad\forall x\not\in \Omega_R.
		\end{equation*}
		
		Using \eqref{inftyinfty}, fix $R\geq R_0$ such that
		\begin{equation}\label{minoracao_de_g}
			g(x)d(x,\Omega)^{d+s}\geq \tfrac{C_0}{s}2\mu_s |\Omega_{R_0}|\|g\|_{L^\infty(\Omega_{R_0})},\qquad 	\forall x\not\in \Omega_{R}	\end{equation}
		and
		\begin{equation*}
			\tfrac{C_0}{s}\frac{2\mu_s |\Omega_{R}|}{d(x,\Omega)^{d+s}}\leq 1.
		\end{equation*}
		
		Let $k\geq \|g\|_{L^\infty(\Omega_R)}$	and consider  $h:\R^d\longrightarrow\R$ 
		\begin{equation}\label{h}
			h(x)=\begin{cases}
				g(x)&\text{ if }x\in\Omega_R,\\
				k&\text{ otherwise.}
			\end{cases}
		\end{equation}
		Then $h$ satisfies assumption \eqref{g_*} and $\K_g^s=\K_h^s$. Indeed, if $u\in\K_g^s$ then, for $x\not\in \Omega_{R}$, 
		\begin{equation*}
			|D^s u(x)|\leq 
			\tfrac{C_0}{s}\frac{2\mu_s |\Omega_{R}|}{d(x,\Omega)^{d+s}}\|g\|_{L^\infty(\Omega_R)}\leq 	\tfrac{C_0}{s}\frac{2\mu_s |\Omega_R|}{R^{d+s}}\|g\|_{L^\infty(\Omega_R)}\leq k.
		\end{equation*}
		
		Reciprocally, if $u\in\K_h$ then, for $x\not\in \Omega_{R}$, we have  $x\not\in \Omega_{R_0}$ and then, as $ \|h\|_{L^\infty(\Omega_{R_0})}=\|g\|_{L^\infty(\Omega_{R_0})}$,
		\begin{equation*}
			|D^s u(x)|  \leq  	\tfrac{C_0}{s}\frac{2\mu_s |\Omega_{R_0}|}{d(x,\Omega)^{d+s}}\|g\|_{L^\infty(\Omega_{R_0})}
		\end{equation*}
		and then, using \eqref{minoracao_de_g}, $ |D^s u(x)|\leq  g(x)$. 
	\end{proof}
	\begin{remark} 
		Assuming only \ref{gweak}, Theorem \ref{teo3.7} remains valid by taking $R>0$ sufficiently large and replacing $\|g_1-g_2\|_{L^\infty(\R^d)}$ by  $\|g_1-g_2\|_{L^\infty(\Omega_R)}$  in \eqref{dep_cont}. 
		
		This is true because in the proof of the last  proposition for $g_1$ and $g_2$, we can define $h_1$ and $h_2$ as in \eqref{h} with the same $R$ and $h_1$ equal to $h_2$ outside $\Omega_R$.
	\end{remark}
	
	\begin{remark}
		If we assume \eqref{matrix}, $\bs b=\bs d=0$ and  $c\geq c_*>0$ instead \eqref{coercive}, keeping the other assumptions in Theorem \ref{teo3.7}, we still have a weaker continuous dependence result, replacing 
		$\|u_1-u_2\|^p_{C^{0,\gamma}(\overline\Omega)}+\|u_1-u_2\|^2_{H^s_0(\Omega)}$ by $\|u_1-u_2\|^2_{L^2(\Omega)}$ in \eqref{dep_cont}. In particular, with these assumptions we also have uniqueness of solution for the variational inequality.
	\end{remark}
	
	\section{Transport potentials and densities}
	\label{section4}
	
	In this section we consider the Lagrange multiplier problem for $0<s\le1$, associated with bounded $s$-gradient constraints: find the generalised transport potential-density pair $(u,\lambda)\in\Lambda^{s,\infty}_0(\Omega)\times L^\infty(\R^d)'$, such that
	\begin{subequations}\label{ml}
		\begin{align}
			&\mathscr L^s(u,v)+\bsl\lambda D^su,D^s v\bsr=[F,v]_s,\quad \forall v\in\Lambda^{s,\infty}_0(\Omega)\label{mla} \\
			&|D^su|\le g\text{ a.e. in }\R^d,\quad\lambda\ge0\quad \text { and } \quad \lambda(|D^su|-g)=0 \quad \text{ in } L^\infty(\R^d)'.\label{mlb}
		\end{align}
	\end{subequations}
	
	In the case $s=1$ the solution $(u,\lambda)$ is to be found in $W^{1,\infty}_0(\Omega)\times L^\infty(\Omega)'$ and the test functions $v$ in $W^{1,\infty}_0(\Omega)$, since $D^1=D$ is the classical gradient.
	
	Here $\bsl\,\cdot\,,\,\cdot\,\bsr$ denotes the duality pairing between $\bs L^\infty(\R^d)'$ and $\bs L^\infty(\R^d)$ and we set, for $\bs\varphi\in \bs L^\infty(\R^d)$,
	\begin{equation*}
		\bsl\lambda\bs\varphi,\bs\xi\bsr=\langle\lambda,\bs\varphi\cdot\bs\xi\rangle,\quad\forall\bs\xi\in\bs L^\infty(\R^d),
	\end{equation*}
	where $\langle\,\cdot\,,\,\cdot\,\rangle$ is the duality pairing between $L^\infty(\R^d)'$ and $L^\infty(\R^d)$.
	\begin{theorem}\label{teo4.1}
		Assume $g$ satisfies \eqref{g_*}, $\mathscr L^s$ is given by \eqref{bilinear} with the assumptions \eqref{matrix}, \eqref{assumptions2} and $F\in \Lambda^{s,\infty}_0(\Omega)'$ is given by \eqref{fs} if $0<s<1$ (respectively $F\in W^{1,\infty}_0(\Omega)'$, if $s=1$) with the assumption \eqref{assumptions_s2}. 
		Then problem \eqref{ml} has a solution $(u,\lambda)\in \Lambda^{s,\infty}_0(\Omega)\times L^\infty(\R^d)'$, for any $0<s<1$ (respectively in $W^{1,\infty}_0(\Omega)\times L^\infty(\Omega)'$ if $s=1$), and $u$ solves the variational inequality \eqref{iv}.
	\end{theorem}
	
	The proof of this existence theorem is obtained by a suitable penalisation of the $s$-gradient constraint, combined with an elliptic nonlinear regularisation, and by a weak stability property of the generalised formulation \eqref{ml} given by the following theorem.
	
	Consider for $0<\nu<1$ the family of solutions $(u_\nu,\lambda_\nu)\in \Lambda^{s,\infty}_0(\Omega)\times L^\infty(\R^d)'$, if $0<s<1$ (respectively, $W^{1,\infty}_0(\Omega)\times L^\infty(\Omega)'$ if $s=1$) 
	\addtocounter{equation}{-1}\begin{subequations}	
		\begin{align}
			&\mathscr L_\nu^s(u_\nu,v)+\bsl\lambda_\nu D^su_\nu,D^s v\bsr=[F_\nu,v]_s,\quad \forall v\in\Lambda^{s,\infty}_0(\Omega)
			\label{mlxa}	\tag*{(\theequation a)$_{\nulab}$}\\
			&|D^su_\nu|\le g_\nu\text{ a.e. in }\R^d,\quad\lambda_\nu\ge0\text{ and }\lambda_\nu(|D^su_\nu|-g_\nu)=0\ \text{ in }L^\infty(\R^d)',
			\label{mlxb}	\tag*{(\theequation b)$_{\nulab}$}
		\end{align}
	\end{subequations}
	where $\mathscr L_\nu^s$ and $F_\nu$ are defined by the \eqref{bilinear} and \eqref{assumptions1} with data $A_\nu$, $\bs b_\nu$, $\bs d_\nu$, $c_\nu$ and $f_{\#,\nu}$, $\bs f_\nu$, respectively.
	
	\begin{theorem}\label{teo4.2}
		Suppose the functions $A_\nu$, $\bs b_\nu$, $\bs d_\nu$, $c_\nu$, $f_{\#,\nu}$, $\bs f_\nu$ and $g_\nu$, for each $\nu$, $0<\nu<1$, satisfy \eqref{g*} \eqref{matrix}, \eqref{assumptions2}, \eqref{assumptions_s2} and \eqref{g_*} and have limit functions as $\nu\longrightarrow 0$: 
		\begin{subequations}\label{abnu}
			\begin{align}
				\label{abnuA}	& A_\nu\longrightarrow A\ \text{in $L^{p_1}(\R^d)^{d^2}$},\quad c_\nu \longrightarrow c \ \text{in $L^1(\Omega)$},\\
				\label{abnubd}		&\bs b_\nu\longrightarrow \bs b\ \text{in $\bs L^1(\Omega)$},\quad \bs d_\nu \longrightarrow \bs d \ \text{in $\bs L^1(\Omega)$},\\
				\label{abnuff}		& f_{\#\nu}\longrightarrow f_\#\ \text{in $L^1(\Omega)$},\quad \bs f_\nu \longrightarrow \bs f \ \text{in $\bs L^{q_1}(\R^d)$},\\
				\label{abnug}		& g_\nu\longrightarrow g\ \text{in $L^\infty(\R^d)$}.
			\end{align}
		\end{subequations}
		
		Then, if $(u_\nu,\lambda_\nu)$ solves {\em \ref{mlxa}\ref{mlxb}}, there is a generalised sequence, still denoted by $\nu\longrightarrow 0$, and a solution $(u,\lambda)$ to \eqref{mla}\eqref{mlb} such that 
		\begin{subequations}\label{solutionabnu}
			\begin{align}
				\label{solutionabnua}&u_\nu\longrightarrow u\text{ in }C^{0,\alpha}(\overline\Omega)\text{-strong}\\
				& D^su_\nu\lraup D^su\text{ in }L^\infty(\Omega)\text{-weak}^*,\\
				\label{solutionabnub}&\lambda_\nu\lraup\lambda\text{ in }L^\infty(\R^d)'\text{-weak}^*,
			\end{align}
		\end{subequations}
		where $0<\alpha<s\le1$, with the convention $\Lambda^{s,\infty}_0(\Omega)=W^{1,\infty}(\Omega)$ in the case $s=1$.
	\end{theorem}
	\begin{proof}
		Since $|D^su_\nu|\le g_\nu\le g^*$ a.e. in $\R^d$, for all $0<\nu<1$, by recalling \eqref{2.21} and \eqref{Morrey}, respectively, we have the {\em a priory} estimates
		\begin{equation}\label{estimates_nu}
			\|D^s u_\nu\|_{\bs L^\infty(\R^d)}\le g^*,\quad\text{and}\quad\|u_\nu\|_{C^{0,\beta}(\overline\Omega)}\le C_\beta,\ \text{ for all }\ 0<\beta<s\le1.	
		\end{equation}
		
		Taking $v=u_\nu$ in \ref{mlxa} we get
		\begin{align}\label{4.11}
			\begin{aligned}
				\langle\lambda_\nu,|D^s u_\nu|^2\rangle&=\bsl\lambda_\nu D^su_\nu,D^su_\nu\bsr=[F_\nu,u_\nu]_s-\mathscr L^s_\nu(u_\nu,u_\nu)\\
				&\le C_1,\quad\text{for all } \sigma<s\le1,
			\end{aligned}
		\end{align}
		where $C_1>0$ is a constant dependent on $g^*$, $C_\beta$ and a common bound of the $L^{p_1}$, $L^1$, and $L^{q_1}$ norms of $(A_\nu)_{\nu}$, $(\bs b_\nu)_{\nu}$, $(\bs d_\nu)_{\nu}$, $(c_\nu)_{\nu}$, $( f_{\#\nu})_{\nu}$ and ${(\bs f_\nu)}_{\nu}$, respectively, independent of $\nu$.
		
		Observing that, as $(u_\nu,\lambda_\nu)$ solves problem \ref{mlxa}\ref{mlxb}, we have $\lambda_\nu(|D^su_\nu|-g_\nu)=0$ in $L^\infty(\R^d)'$, which, multiplying by $|D^su_\nu|+g_\nu$, implies
		\begin{equation}\label{4_7}
			\langle\lambda_\nu,|D^su_\nu|^2\rangle=\langle\lambda_\nu,g_\nu^2\rangle.
		\end{equation}
		Using the assumption $g_\nu\ge g_*$  and $\lambda_\nu\ge0$ we have
		\begin{align}\label{4.12}
			\begin{aligned}
				\|\lambda_\nu\|_{L^\infty(\R^d)'} & = \sup_{\|\zeta\|_{L^\infty(\R^d)}\le1 }|\langle\lambda_\nu,\zeta\rangle|\leq  \sup_{\|\zeta\|_{L^\infty(\R^d)}\le1 }\langle\lambda_\nu,|\zeta|\rangle\leq\langle\lambda_\nu,1\rangle\leq\langle \lambda_\nu,\tfrac{g_\nu^2}{g_*^2}\rangle\\
				&=\tfrac{1}{g_*^2}\langle\lambda_\nu,|D^su_\nu|^2\rangle
				\leq \tfrac{C_1}{g_*^2},\quad\text{ by \eqref{4_7} and \eqref{4.11}}.
			\end{aligned}
		\end{align}
		
		As a consequence, letting $\bs\Psi_\nu=\lambda_\nu D^su_\nu$, we also have
		\begin{equation}\label{4.13}	
			\|\bs\Psi_\nu\|_{\bs L^\infty(\R^d)'} =\sup_{\|\bs \xi\|_{\bs L^\infty(\R^d)}\le1}|\langle\lambda_\nu,D^s u_\nu\cdot\bs\xi\rangle|\le
			\|\lambda_\nu\|_{L^\infty(\R^d)'}\|D^su_\nu\|_{\bs L^\infty(\R^d)}\leq\tfrac{C_1 g^*}{g_*^2}.
		\end{equation}

		Then, by the above estimates, we may choose some generalised sequence $\nu\longrightarrow0$, such that: i) \eqref{solutionabnua} holds with $\alpha<\beta$ for some $u\in C^{0,\alpha}(\overline\Omega)$  and \eqref{solutionabnub} holds, using the distributional nature of $D^s$, since \eqref{estimates_nu} is satisfied and  ii) by the Banach-Alaoglu-Bourbaki theorem, \eqref{solutionabnub} holds for some $\lambda\in L^\infty(\R^d)'$ (by estimate \eqref{4.12}) and there exists a $\bs\Psi\in\bs L^\infty(\R^d)'$ (by \eqref{4.13}) with
		\begin{equation}\label{psinu}
			\bs\Psi_\nu\underset{\nu}\lraup\bs\Psi\quad\text{ in }\bs L^\infty(\R^d)'.
		\end{equation}
		
		Letting $\nu\longrightarrow0$ in \ref{mlxa}, by the assumptions \eqref{abnuA}, \eqref{abnubd}, \eqref{abnuff} and the convergences \eqref{solutionabnua} and \eqref{solutionabnub}, we conclude that $(u,\bs\Psi)$ solves:
		\begin{equation*}
			\mathscr L^s(u,v)+\bsl\bs\Psi,D^sv\bsr=[F,v]_s,\qquad\forall v\in \Lambda^{s,\infty}_0(\Omega).
		\end{equation*}
		
		Note that $\lambda_\nu\ge0$ implies $\lambda\ge0$. Given any measurable set $\omega\subset\R^d$ with finite measure, taking $\bs\xi\in\bs L^1(\R^d)$, defined by $\bs\xi=\frac{D^su}{|D^su|}$ if $x\in\omega\cap\{|D^su|\neq0\}$ and $\bs\xi=0$ elsewhere, since $D^su_\nu\underset\nu\lraup D^su$ in $\bs L^\infty(\R^d)\text{-weak}^*$, we have
		\begin{equation}\label{pertence_convexo}
			\int_\omega g_\nu\geq\int_{\omega}|D^s u_\nu|\ge\int_\omega D^su_\nu\cdot\bs\xi\underset\nu\longrightarrow\int_{\R^d}D^su\cdot\bs\xi=\int_\omega|D^su|,
		\end{equation} 
		and so  $|D^s u|\le g$ a.e. in $\R^d$, by \eqref{abnug} and the arbitrariness of $\omega\subset\R^d$. Then, in order to complete the proof, it remains to show that
		\begin{equation}\label{psilambda}
			\bs\Psi=\lambda D^su\qquad\text{ in }\bs L^\infty(\R^d)'
		\end{equation}
		and
		\begin{equation}\label{lambdag}
			\lambda|D^su|=\lambda g\qquad\text{ in } L^\infty(\R^d)',
		\end{equation}
		or equivalently, using the assumption \eqref{g_*},
		\begin{equation*}
			\langle\lambda(g-|D^su|),\varphi\rangle=\langle\lambda,(g^2-|D^su|^2)\tfrac{\varphi}{g+|D^su|}\rangle=0,\quad\forall\varphi\in L^\infty(\R^d).
		\end{equation*}
		
		Then we observe that, recalling $|D^su|\le g$ and using \eqref{4.12},
		\begin{align}\label{4.20}
			\begin{aligned}
				\tfrac{1}{2}	\langle\lambda_\nu ,|D^s (u_\nu -u)|^2\rangle&=\tfrac{1}{2}\Big(\langle\lambda_\nu ,|D^s u_\nu |^2\rangle-2\langle\lambda_\nu ,D^s u_\nu \cdot D^s u\rangle+\langle\lambda_\nu ,|D^s u|^2\rangle\Big)\\
				&\leq \langle\lambda_\nu ,|D^s u_\nu |^2\rangle-\langle\lambda_\nu ,D^s u_\nu \cdot D^s u\rangle+\tfrac12\langle\lambda_\nu,g^2-g_\nu^2\rangle\\
				&\le\langle\lambda_\nu D^s u_\nu , D^s (u_\nu -u)\rangle+\tfrac12\|\lambda_\nu\|_{L^\infty(\R^d)'}\|g^2-g_\nu^2\|_{L^\infty(\R^d)}\\
				&\le[F_\nu ,u_\nu -u]_s-\opL^s_\nu  (u_\nu ,u_\nu -u)+\tfrac{1}{2}\tfrac{C_1}{g_*^2}\|g^2-g_\nu^2\|_{L^\infty(\R^d)}.
			\end{aligned}
		\end{align}

		We have $[F_\nu,u_\nu-u]_s\underset\nu\longrightarrow0$ and $\opL^s_\nu  (u_\nu,u)\underset\nu\longrightarrow\opL^s(u,u)$, while, on the other hand,
		\begin{equation}\label{4.21}
			\varliminf_\nu\opL^s_\nu(u_\nu,u_\nu)\ge\opL^s(u,u)
		\end{equation}
		since, arguing as in \eqref{liminf} we have that, using \eqref{abnuA},
		\begin{align*}
			\varliminf_\nu\int_{\R^d}A_\nu D^su_\nu\cdot D^su_\nu&\ge\varliminf_\nu\int_{\R^d} AD^su_\nu\cdot D^su_\nu+\lim_\nu\int_{\R^d}(A_\nu-A)D^su_\nu\cdot D^su_\nu\\
			&\ge\int_{\R^d}AD^su\cdot D^su
		\end{align*}
		as
		\begin{equation*}
			\Big|\int_{\R^d}(A_\nu-A)D^su_\nu\cdot D^su_\nu\Big|\le(g^*)^2\|A_\nu-A\|_{L^1(\R^d)^{d^2}}\underset\nu\longrightarrow0
		\end{equation*}
		and, using the convergences \eqref{solutionabnua} with \eqref{abnuA} and \eqref{abnubd},
		\begin{equation*}
			\int_\Omega\bs b_\nu\cdot D^su_\nu(u_\nu-u)+\bs d_\nu u_\nu\cdot D^s(u_\nu-u)+c_\nu u_\nu(u_\nu-u)\underset\nu\longrightarrow0.
		\end{equation*}
		
		Hence we conclude that
		\begin{equation*}
			0\le\lim_\nu\langle\lambda_\nu,|D^s(u_\nu-u)|^2\rangle\le\varlimsup_\nu\langle\lambda_\nu,|D^s(u_\nu-u)|^2\rangle\le0.
		\end{equation*}
		By the H\"older inequality (see Proposition \ref{desigHolder}),
		\begin{align*}
			|\langle \lambda_\nu , D^s(u_\nu -u)\cdot \bs\xi\rangle| &\le \langle \lambda_\nu , |D^s(u_\nu -u)|| \bs\xi|\rangle\le\langle\lambda_\nu ,|D^s(u_\nu -u)|^2\rangle^\frac12\langle\lambda_\nu ,|\bs\xi|^2\rangle^\frac12\\
			&\leq \langle\lambda_\nu ,|D^s(u_\nu -u)|^2\rangle^\frac12\|\lambda_\nu\|_{L^\infty(\Omega)'}^\frac{1}{2}\|\bs \xi\|_{\bs L^\infty(\R^d)}
		\end{align*}
		which by \eqref{4.12} yields
		\begin{equation}\label{423}
			\lim_\nu\langle\lambda_\nu,D^s(u_\nu-u)\cdot\bs\xi\rangle=0,\qquad\forall\bs\xi\in\bs L^\infty(\R^d).
		\end{equation}
		
		Now, recalling \eqref{psinu}, \eqref{psilambda} follows now from \eqref{423}, since
		\begin{align}\label{4.24}
			\begin{aligned}
				\bsl\bs\Psi,\bs\xi\bsr&=\lim_\nu  \bsl\bs\Psi_\nu ,\bs\xi\bsr=\lim_\nu \langle \lambda_\nu , D^s u_\nu \cdot\bs\xi\rangle=\lim_\nu \langle \lambda_\nu , D^s u\cdot\bs\xi\rangle\\
				&=\langle \lambda, D^s u\cdot\bs\xi\rangle=\bsl\lambda D^s u,\bs\xi\bsr\qquad\forall\bs\xi\in\bs L^\infty(\R^d).
			\end{aligned}
		\end{align}
		
		Using \eqref{423} and \eqref{4.24} with $\bs\xi=D^su_\nu$ we obtain $\lim_\nu\bsl\lambda_\nu D^su_\nu,D^s(u_\nu-u)\bsr=0$ and 
		\begin{equation}\label{4.25}
			\langle\lambda,g^2\rangle=\lim_\nu\langle\lambda_\nu,g_\nu^2\rangle=\lim_\nu\bsl\lambda_\nu D^su_\nu,D^su_\nu\bsr=\lim_\nu\bsl\lambda_\nu D^su_\nu,D^su\bsr=\langle\lambda,|D^su|^2\rangle,
		\end{equation}
		which implies $\langle\lambda(g^2-|D^su|^2),1\rangle=0$.
		
		Finally, \eqref{lambdag} follows again by the H\"older inequality for charges, with an arbitrary $\varphi\in L^\infty(\R^d)$,
		\begin{align}\label{4.26}
			\begin{aligned}
				|\langle\lambda(g-|D^su|),\varphi\rangle| & =\big|\langle\lambda(g^2-|D^su|^2),\tfrac{\varphi}{g+|D^su|}\rangle\big|\\
				& \le\langle\lambda(g^2-|D^su|^2),1\rangle^\frac12\langle\lambda(g^2-|D^su|^2),\tfrac{|\varphi|^2}{(g+|D^su|)^2}\rangle^\frac12=0.
			\end{aligned}
		\end{align}
	\end{proof}
	\begin{proof} (of Theorem \ref{teo4.1})
		It can be obtained with the following approximating problem. Let $0<\eps<1$ and fix $q>1+\frac{d}s>2$, so that $\Lambda ^{s,r}_0(\Omega)\subset C(\overline\Omega)$, for $0<s\le1$ and $r=q-1>\frac{d}{s}$, recalling the convention $\Lambda^{s,r}_0(\Omega)=W^{1,r}_0(\Omega)$ if $s=1$. We firstly consider $\opL^s$ and $F$ defined by \eqref{bilinear} and \eqref{fs} under the assumption \eqref{matrix} and, letting $r'=\tfrac{r}{r-1}$,
		\begin{equation}\label{dados_regulares}
			A\in L^\infty(\R^d)^{d^2},\quad\bs b,\bs d\in \bs L^{r'}(\Omega),\quad \bs f\in \bs L^{q'}(\R^d),\quad 
			c,\ f_{\#}\in L^1(\Omega),
		\end{equation}
		with which we shall prove the existence of a solution $(u,\lambda)\in\Lambda^{s,\infty}_0(\Omega)\times L^\infty(\R^d)'$, $0<s\le1$ of \eqref{ml}.
		
		The approximating problem is given, for $\eps>0$, by
		\begin{equation}\label{prob_ap}
			u_\eps\in\Lambda^{s,q}_0(\Omega):\quad\opL^s(u_\eps,v)+[\widetilde k_\eps(u_\eps)+\eps D_q^s u_\eps,v]=[F,v],\qquad \forall v\in\Lambda^{s,q}_0(\Omega),
		\end{equation}
		where $\eps>0$ and the penalisation operator
		\begin{equation*}
			[\widetilde k_\eps(w),v]=\int_{\R^d}k_\eps(|D^sw|-g)D^sw\cdot D^sv,\qquad\forall v,w\in\Lambda^{s,q}_0(\Omega),
		\end{equation*}
		is defined with the continuous monotone function
		\begin{equation*}
			k_\eps(t)=0\ \text{ for }t\le0,\quad k_\eps(t)=e^\frac{t}\eps-1\ \text{ for } 0<t\le\tfrac1\eps\quad\text{ and }\quad k_\eps(t)=e^\frac1{\eps^2}-1\ \text{ for }t\ge\tfrac1\eps
		\end{equation*}
		and the nonlinear elliptic regularisation is given by
		\begin{equation*}
			[\eps D^s_qw,v]=\eps\int_{\R^d}|D^sw|^{q-2}D^sw\cdot D^sv,\qquad\forall v,w\in\Lambda^{s,q}_0(\Omega).
		\end{equation*}
		
		As in Lemma \ref{klimitado}, the nonlinear operator $[B_\eps u,v]=\opL^s(u,v)+[\widetilde k_\eps(u)+\eps D_q^s u_\eps,v]$ is easily seen to be pseudo-monotone in $\Lambda^{s,q}_0(\Omega)$ and also coercive (see \cite{Lions1969} or \cite{Roubicek2013}), since, setting $\|v\|=\|D^sv\|_{\bs L^q(\R^d)},$
		\begin{align*}
			\frac{[B_\eps v,v]}{\|v\|}&=\frac1{\|v\|}\bigg(\int_{\R^d}\big(\eps|D^s_q v|^q+AD^sv\cdot D^sv+k_\eps(|D^sv|-g)|D^sv|^2\big)\\
			&\quad+\int_\Omega\big(v(\bs d+\bs b)\cdot D^sv+cv^2\big)\bigg)\\
			&\ge \eps\|D^sv\|^{q-1}_{\bs L^q(\Omega)}-\tfrac{\|v\|_{L^\infty(\Omega)}}{\|v\|}
			\Big(\|\bs d+\bs b\|_{\bs L^{r'}(\Omega)}\|D^sv\|_{\bs L^r(\Omega)}+\|c\|_{L^1(\Omega)}\|v\|_{L^\infty(\Omega)}\Big)\\
			&\ge \|v\|\big(\eps\|v\|^{q-1}-\widetilde C_q\big)\longrightarrow\infty\quad\text{ as }\|v\|\longrightarrow\infty,
		\end{align*}
		with $\widetilde C_q=C_{r,q}C_q\big(\|\bs d+\bs b\|_{\bs L^{r'}(\Omega)}+C_q\|c\|_{L^1(\Omega)}\big)$, where $C_{r,q}$ is given by \eqref{PoincareSobolevLambda} and $C_q>0$ is given by the continuous embedding \eqref{2.19}, i.e., such that $\|v\|_{L^\infty(\Omega)}\le C_q\|D^sv\|_{\bs L^q(\R^d)}$, $v\in \Lambda^{s,q}_0(\Omega)$.
		
		Then, by the theory of pseudo-monotone and coercive operators (see \cite{Lions1969} or \cite{Roubicek2013}, for instance), since $F\in\Lambda^{s,q}_0(\Omega)'$, there exists $u_\eps\in\Lambda^{s,q}_0(\Omega)$ solving \eqref{prob_ap}. Taking $v=u_\eps$ in \eqref{prob_ap} and setting $\widehat k_\eps=k_\eps(|D^su_\eps|-g)$,
		\begin{align}\label{4.30}
			\begin{aligned}
				\int_{\R^d}\widehat k_\eps |D^s   u_{\eps}|^2+\eps & \int_{\R^d}|D^s   u_{\eps}|^q\le
				\widetilde C_r\|D^su_\eps\|_{\bs L^r(\R^d)}^2\\
				&+\Big(C_r\|f_\#\|_{L^1(\Omega)}+\|\bs f\|_{\bs L^{q'}(\R^d)}\Big)\|D^su_\eps\|_{\bs L^q(\R^d)}\\
				&\le \widetilde C_r'\|D^su_\eps\|_{\bs L^r(\R^d)}^2\le \widetilde C_r'C^2_{r,q}\|D^su_\eps\|_{\bs L^q(\R^d)}^2
			\end{aligned}
		\end{align}
		with $\widetilde C_r=C_r\big(\|\bs d+\bs b\|_{\bs L^{q'}(\Omega)}+C_r\|c\|_{L^1(\Omega)}\big)$ by assuming, without loss of generality, that \linebreak$\|D^su_\eps\|_{\bs L^r(\R^d)}\ge1$, which will allow the proof of the following {\em a priori} estimates independent of $\eps$, $\eps$ sufficiently small,
		\begin{equation}\label{4.31}
			\|D^su_\eps\|_{\bs L^r(\R^d)}\le C,
		\end{equation}
		\begin{equation}\label{4.32}
			\|\widehat k_\eps|D^su_\eps|^2\|_{ L^1(\R^d)}\le C\quad\text{ and }\quad \|k_\eps\|_{ L^1(\R^d)}\le C,
		\end{equation}
		\begin{equation}\label{4.33}
			\|\widehat k_\eps D^su_\eps\|_{\bs L^\infty(\R^d)'}\le C.
		\end{equation}
		
		From \eqref{4.30}, there exists $C>0$, independent of $\eps$, such that $\|D^su_\eps\|_{\bs L^q(\R^d)}\le C\eps^{-1/(q-2)}$ and consequently also
		\begin{equation}\label{4.34}
			g_*^2	\|\widehat k_\eps\|_{ L^1(\R^d)}\le\|\widehat k_\eps|D^su_\eps|^2\|_{ L^1(\R^d)}\le C\eps^{-\frac{2}{q-2}},
		\end{equation}
		since, as $\widehat k_\eps=0$ if $|D^su_\eps|<g$, we have $\widehat k_\eps|D^su_\eps|^2\ge g^2\widehat k_\eps\ge g_*^2\widehat k_\eps$.
		
		Now we split $\R^d$ in two subsets,
		\begin{equation}\label{UepsVeps}
			U_{\eps} = \{x\in \R^d: |D^s   u_{\eps}|-g\leq \sqrt\eps\}\quad\text{and}\quad V_{\eps}=\R^d\setminus U\eps
		\end{equation}
		and we observe that, as $k_\eps $ is a monotone function, in $V_{\eps}$ we have $\widehat k_\eps =k_\eps (|D^s u_{\eps }|-g)\geq  k_\eps (\sqrt\eps)=e^\frac1{\sqrt\eps}-1$ and 
		\begin{align}
			\label{veps}		&|V_{\eps}|=\int_{V_{\eps}}1\leq \int_{V_{\eps}}\frac\ked{e^\frac1{\sqrt\eps}-1}\le\tfrac1{e^\frac1{\sqrt\eps}-1}\int_{\R^d}\ked\le \tfrac{C}{\eps^{\frac2{q-2}}\big(e^\frac1{\sqrt\eps}-1\big)}\underset{\eps\rightarrow0}{\longrightarrow}0,\\[2mm]
			\nonumber		&\int_{V_\eps}|D^s  u_{\eps}|^{r}\le\bigg(\int_{\R^d}|D^s  u_{\eps}|^{q}\bigg)^\frac{r}{q}|V_\eps|^\frac{q-r}{q}\le C\Big(\eps^{-\frac1{q-2}}\Big)^\frac{r}{q}\bigg(\tfrac{1}{\eps^{\frac2{q-2}}\big(e^\frac1{\sqrt\eps}-1\big)}\bigg)^\frac{q-r}{q} \underset{\eps\rightarrow0}{\longrightarrow}0.       
		\end{align}
		
		Given $R>0$,
		\begin{align*}
			&\int_{U_{\eps}\cap \Omega_R}|D^s  u_{\eps}|^{r}\leq \int_{U_{\eps}\cap \Omega_R}(g+\sqrt\eps)^{r}\le \big(g^*+1\big)^{r}|\Omega_R|,\\[2mm]
			&\int_{U_{\eps}\setminus \Omega_R}|D^s  u_{\eps}|^r\leq 
			C(R)\|u_{\eps}\|_{L^1(\Omega)}^r\leq 
			C(R)C_{r,1}^r   \|D^s  u_{\eps}\|_{L^r(\R^d)}^r,
		\end{align*}
		using \eqref{trudinger} and \eqref{desigualdadeomegaR}, with $C$ representing different constants. Then, for $\eps$ small enough,
		\begin{equation*}
			\int_{\R^d}|D^s  u_{\eps}|^r\le \big(g^*+1\big)^{r}|\Omega_R|+ C(R)C_{r,1}^r   \|D^s  u_{\eps}\|_{L^r(\R^d)}^r+1.
		\end{equation*}
		
		Choosing $R_0>0$ such that $C(R_0)C_{r,1}^r\leq \frac{1}{2}$ we get \eqref{4.31} from 
		\begin{equation*}
			\|D^s  u_{\eps}\|_{\bs L^r(\R^d)}\le 2^{\frac{1}{r}}\big((g^*+1)^{r}|\Omega_{R_0})|+1\big)^{\frac{1}{r}}.
		\end{equation*}
		
		Hence, also from \eqref{4.30} we immediately obtain that $\|\widehat k_\eps |D^s  u_{\eps}|^2\|_{L^1(\R^d)} \le C$, and \eqref{4.32} follows from the first inequality in \eqref{4.34}.
		
		As a consequence, \eqref{4.33}  now easily follows from \eqref{4.32}:
		\begin{align*}
			\|\widehat k_\eps  D^s   u_{\eps}\|_{\bs L^\infty(\R^d)'}&=\sup_{
				\|\bs\xi\|_{\bs L^\infty(\R^d)}\le1}\bigg|\int_{\R^d}\ked D^s  u_{\eps}\cdot\bs \xi\bigg|\\
			&\leq \int_{\R^d}\ked^{\frac{1}{2}}\ked^{\frac{1}{2}}|D^s  u_{\eps}|\leq \|\ked\|_{L^1(\R^d)}^{\frac{1}{2}}	\|\ked|D^s  u_{\eps}|^2\|_{L^1(\R^d)}^{\frac{1}{2}}.
		\end{align*}
		
		By compactness, from the estimates \eqref{4.31}, \eqref{4.32} and \eqref{4.33}, using the Rellich-Kondrachov  and the  Banach-Alaoglu-Bourbaki theorems, there exist $u\in\Lambda^{s,r}_0(\Omega)\cap C^{0,\gamma}(\overline{\Omega})$, $\lambda\in L^\infty(\R^d)'$ and $\bs\Psi\in\bs L^\infty(\R^d)'$ and a generalised sequence $\eps\rightarrow0$ such that
		\begin{align*}
			D^su_\eps\underset{\eps}\lraup D^su \text{ in } \ \bs L^r(\R^d)\text{-weak},\quad
			u_\eps\underset{\eps}\longrightarrow u \text{ in } \  C^{0,\gamma}(\overline{\Omega}) \text{ strong},\label{4.36}\\
			\ked\underset{\eps}\lraup \lambda \text{ in } \  L^\infty(\R^d)'\text{-weak}^*,\quad
			\ked D^su_\eps\underset{s}\lraup \bs\Psi\ \text{ in } \ \bs L^\infty(\R^d)'\text{-weak}^*.
		\end{align*}
		
		Now, arguing as in the proof of Theorem \ref{teo4.2}, we show that $u\in\K^s_g$, $\bs\Psi=\lambda D^su$ and $(u,\lambda)$ satisfies \eqref{ml}.
		
		Firstly we conclude that $|D^su|\leq g$ a.e. in $\R^d$, from
		\begin{align*}
			\int_{\R^d}(|D^su |-g)^+&\le\varliminf_{\eps\rightarrow 0}	\int_{\R^d}(|D^s   u_\eps|-g-\sqrt\eps)^+
			=\varliminf_{\eps\rightarrow 0}	\int_{V\eps}(|D^su_\eps|-g-\sqrt\eps)\\
			&\leq\varliminf_{\eps\rightarrow 0}	\int_{V\eps}|D^s   u_\eps|
			\leq\varliminf_{\eps\rightarrow 0}\|D^s   u_\eps\|_{L^r(\R^d)}|V_\eps|^{\frac{1}{r'}}=0,
		\end{align*}
		since $\xi\mapsto (|\xi|-g)^+$ is a convex lower semicontinuous function and $V_\eps$, defined in \eqref{UepsVeps}, has vanishing measure as $\eps\rightarrow 0$ by \eqref{veps}.
		
		Observing that
		\begin{equation*}
			\bigg|\int_{\R^d}|D^su_\eps|^{q-2}D^su_\eps\cdot D^sv\bigg|\leq \|D^su_\eps\|^r_{\bs L^r(\R^d)}\|D^s v\|_{\bs L^\infty(\R^r)},
		\end{equation*}
		taking the generalised limit in \eqref{4.24} with an arbitrarily $v\in \Lambda_0^{s,\infty}(\Omega)\subset \Lambda_0^{s,q}(\Omega)$, we obtain
		\begin{equation}
			\label{4.38}
			\bsl \bs\Psi,D^sv\bsr =[F,v]-\mathscr{L}^s(u,v),\quad \forall v\in \Lambda_0^{s,\infty}(\Omega),
		\end{equation} 
		and, since $\ked\geq 0$ implies $\lambda\geq 0$, it remains to show that $\bs\Psi=\lambda D^su$ and $\lambda|D^su|=\lambda g$.
		
		Taking $v=u_\eps$ in \eqref{prob_ap} and using \eqref{4.38} and the semicontinuity \eqref{4.21} for $u_\eps\underset{\eps}{\lraup}u$ in $\Lambda^{s,r}_0(\Omega)$, we easily obtain 
		\begin{equation*}
			\varlimsup_{\eps\rightarrow 0}\int_{\R^d}\ked|D^su_\eps|^2\leq \bsl\bs\Psi,D^su\bsr.
		\end{equation*}
		
		Comparing the inequality
		\begin{align}	\label{4.39}
			\begin{aligned}
				0 & \leq 	\varlimsup_{\eps\rightarrow 0}\int_{\R^d}\ked|D^s(u_\eps-u)|^2\\
				& =
				\varlimsup_{\eps\rightarrow 0}\bigg(\int_{\R^d}\ked|D^su_\eps|^2-2\int_{\R^d}\ked D^su_\eps\cdot D^su+\int_{\R^d}\ked|D^su|^2\bigg)	\\
				& \leq \bsl\bs\Psi,D^su\bsr-2 \lim_{\eps\rightarrow 0}\bsl\ked D^su_\eps,D^su\bsr+\lim_{\eps\rightarrow 0}\langle \ked,|D^su|^2\rangle\\
				& \leq \langle \lambda,|D^su|^2\rangle-\bsl\bs\Psi,D^su\bsr,
			\end{aligned}
		\end{align}
		with (recall that $\ked g^2\leq \ked|D^su_\eps|^2$ by definition of $\ked$)	
		\begin{equation*}
			\langle\lambda,|D^su|^2\rangle\leq \langle\lambda,g^2\rangle=
			\lim_{\eps\rightarrow 0}\int_{\R^d}\ked\,g^2	\leq \varlimsup_{\eps\rightarrow 0}\int_{\R^d}\ked|D^su_\eps|^2\leq \bsl\bs\Psi,D^su\bsr,
		\end{equation*}
		we conclude that
		\begin{equation*}
			\langle\lambda,|D^su|^2\rangle= \langle\lambda,g^2\rangle=
			\varlimsup_{\eps\rightarrow 0}\int_{\R^d}\ked|D^su_\eps|^2= \bsl\bs\Psi,D^su\bsr
		\end{equation*}
		and, afterwards from \eqref{4.39}, also
		\begin{equation*}
			\varlimsup_{\eps\rightarrow 0}\int_{\R^d}\ked|D^s(u_\eps-u)|^2=0.
		\end{equation*}
		
		Then $\bs\Psi=\lambda D^su$ since, for an arbitrary $\bs\xi\in\bs L^\infty(\R^d)$,
		\begin{align*}
			|\bsl\bs\Psi-\lambda D^su,\bs\xi\bsr| & =	\lim_{\eps\rightarrow 0}\bigg|\int_{\R^d}\ked D^s(u_\eps-u)\cdot\bs\xi\bigg|\\
			& \leq 
			\varlimsup_{\eps\rightarrow 0}\bigg(\int_{\R^d}\ked|D^s(u_\eps-u)|^2\bigg)^\frac12\|\ked\|_{L^1(\R^d)}^\frac12\|\xi\|_{\bs L^\infty(\R^d)}=0.
		\end{align*}
		
		From \eqref{4.33} it follows $\langle\lambda,|D^su|^2-g^2\rangle=0$ and we conclude, as in \eqref{4.26}, that  
		$$\lambda(|D^su|-g)=0\quad\text{in }L^\infty(\R^d)'.$$
		
		The general case follows by Theorem \ref{teo4.2}, by approximating with solutions of \ref{mlxa}, \ref{mlxb} with data $A_\nu$, $\bs b_\nu$, $\bs d_\nu$, $c$, $f_\#$, $\bs f_\nu$ and $g$ satisfying \eqref{matrix}, \eqref{dados_regulares} and \eqref{g_*} and converging strongly in $L^{p_1}$, $L^1$ and $L^{q_1}$,  for instance by using  $\bs f_\nu=\sup(-\frac{1}{\nu},\inf(\frac{1}{\nu},\bs f))\chi_{B(0,\frac{1}{\nu})}$, where $\chi_{B(0,\frac{1}{\nu})}$ denotes the characteristic function of $B(0,\frac{1}{\nu})$.
		
		Finally, since $u\in\K^s_g$, given $v\in\K^s_g$, taking $v-u$ as test function in \eqref{mla} and noting that
		\begin{align*}
			\bsl\lambda D^su,D^s(v-u)\bsr&=\langle\lambda,D^su\cdot D^s(v-u)\rangle=\langle\lambda, D^su\cdot D^sv-|D^su|^2\rangle\\
			& \leq 	\langle\lambda,|D^s u|(|D^sv|-|D^su|)\rangle	\leq 	\langle\lambda,|D^s u|(g-|D^su|)\rangle\\
			&=\langle\lambda(g-|D^su|),|D^su|\rangle=0,
		\end{align*}
		it is clear that $u$ solves the variational inequality \eqref{iv}, which concludes the proof of Theorem \ref{teo4.1}.
	\end{proof}
	
	\begin{remark}
		Theorem \ref{teo4.2}, as a weak continuous dependence result, generalises Theorem \ref{teo3.5} to the case of degenerate operators, including the case $A\equiv 0$, with $L^1$-data. In fact, if $A$ satisfies \eqref{matrix} and \eqref{assumptions1} holds with the strictly coercive assumption \eqref{coercive}, it is clear that $u$ solving problem \eqref{ml} is unique. On the other hand, the uniqueness of $\lambda$ is an open problem, even in the local case of $s=1$, which was considered first for the Laplacian in \cite{AzevedoSantos2017} in the special case $f_\#\in L^2(\Omega)$ and $\bs f=0$.
	\end{remark}
	\begin{remark}\label{local_local_4}
		As in Section 3, we may assume in Theorems \ref{teo4.1} and \ref{teo4.2} that $g$ and $g_\nu$ satisfy \eqref{gweak} instead 
		\eqref{g_*}, with uniform limit in $\nu$.
	\end{remark}
	
	\section{Localisation of transport densities as $s\rightarrow 1$}
	\label{section5}
	
	In order to consider the generalised convergence of the fractional problem to the local one as $s\rightarrow 1$, for $0<\sigma <s\leq 1$, with $\sigma$ fixed, we consider $(u_s,\lambda_s)\in \Lambda^{s,\infty}_0(\Omega)\times \bs L^\infty(\R^d)'$ such that
	\refstepcounter{equation}
	\begin{align*}
		&\mathscr L^s(u_s,v)+\bsl\lambda_s D^su_s,D^s v\bsr=[F,v]_s,\quad \forall v\in\Lambda^{s,\infty}_0(\Omega)
		\label{5.1.s}	\tag*{(\theequation)$_{\slab}$}\\
		&			\refstepcounter{equation}
		|D^su_s|\le g_s\text{ a.e. in }\R^d,\quad\lambda_s\ge0\text{ and }\lambda_s(|D^su_s|-g_s)=0\ \text{ in }L^\infty(\R^d)', \label{5.2.s}	\tag*{(\theequation)$_{\slab}$}
	\end{align*}
	with the convention $s=1$ corresponds to the local problem $(u,\lambda)\in W^{1,\infty}_0(\Omega)\times L^\infty(\Omega)$, where $D^1=D$ is the classical gradient in the definitions of $\mathscr L$ and $[F,\cdot]_s$ given by \eqref{bilinear} and \eqref{fs}, respectively, and  {\renewcommand\slab{1}\ref{5.2.s}} holds only in $\Omega$.
	
	Here we can also allow a variable threshold $g_s$ under the assumption
	\begin{equation}\label{5.3}
		0<g_*\leq g_s(x)\leq g^*,\quad\text{a.e. $x\in\R^d$, $\sigma<s\leq 1$},
	\end{equation}
	and such that
	\begin{equation}
		\label{5.4}
		g_s\longrightarrow g_1\quad\text{ in }\ L^\infty(\R^d)\ \text{ as } s\rightarrow 1. 
	\end{equation}
	
	The corresponding variational inequality \eqref{iv} now reads as follows
	\refstepcounter{equation}
	\begin{equation}\label{5.5s}
		\tag*{(\theequation)$_{\slab}$}
		u_s\in\K_{g_s}^s:\ \mathscr L^s(u_s,v-u_s)\geq [F,v-u_s]_s,\quad \forall v\in\K_{g_s}^s,
	\end{equation}	
	where the convex set $\K_{g_s}$ is defined by \eqref{kg} with $g_s$, $\sigma <s\leq 1$.
	
	Now we can state the localisation theorem for \ref{5.1.s}-\ref{5.2.s} as $s\rightarrow 1$, which is essentially a variant of the generalised continuous dependence property of Theorem \ref{teo4.2} with the additional difficulty on the variable spaces $\Lambda^{s,\infty}_0(\Omega)$, with $s$, $\sigma<s<1$. For $\zeta\in L^\infty(\R^d)'$, we denote its restriction to $\Omega\subset\R^d$ by $\zeta_\Omega\in L^\infty(\Omega)'$, defined by 
	\begin{equation*}
		\langle \zeta_\Omega,\varphi\rangle=\langle\zeta,\widetilde{\varphi}\rangle,\quad\forall \varphi\in L^\infty(\Omega),
	\end{equation*}
	where $\widetilde{\varphi}$ is the extension of $\varphi$ by zero to $\R^d\setminus\Omega$.
	\begin{theorem}\label{teo5.1}
		For any  $0<\sigma<1$, let $(u_s,\lambda_s)\in\Lambda^{s,\infty}_0(\Omega)\times L^\infty(\R^d)'$ solve {\em \ref{5.1.s}-\ref{5.2.s}} for any $s$, $0<\sigma<s<1$, under the assumptions \eqref{matrix}, \eqref{assumptions2}, \eqref{assumptions_s2} and \eqref{5.3}, \eqref{5.4}. Then, there is a generalised sequence denoted by $s\rightarrow 1$, such that, for any $0<\alpha<1 $,
		\begin{align*}
			&	u_s\underset{s}\rightarrow u\ \text{ in } \ \Lambda^{\sigma,p}_0(\Omega)\cap C^{0,\alpha}(\overline{\Omega}),\\
			&	D^su_s\underset{s}\lraup Du\ \text{ in } \ \bs L^\infty(\R^d)\text{-weak}^*,\\
			&		(\lambda_s)_\Omega\underset{s}\lraup \lambda \text{ in } \  L^\infty(\Omega)'\text{-weak}^*,
		\end{align*}
		where $(u,\lambda)\in W^{1,\infty}_0(\Omega)\cap L^\infty(\Omega)'$ is a solution to the local problem {\em {\renewcommand\slab{1}\ref{5.1.s}}-{\renewcommand\slab{1}\ref{5.2.s}}}, in $\Omega$.
	\end{theorem}
	
	\begin{proof}
		We adapt the steps of the proof of Theorem \ref{teo4.2}: i) {\em a priori} estimates with respect to $s$; ii) existence of limits of generalised sequences, by compactness, and iii) characterization of those limits as solutions of the local
		problem  {\renewcommand\slab{1}\ref{5.1.s}}- {\renewcommand\slab{1}\ref{5.2.s}}.		 	
		For $0<\sigma<s<1$, using the Poincar\'e inequality \eqref{PoincareSobolev}, we have $\frac{\sigma}{C_0}\|u_s\|_{L^\infty(\Omega)}\leq \|D^su_s\|_{\bs L^\infty(\R^d)}$. Then by the assumption \eqref{5.3} we obtain for any $u_s\in\K^s_{g_s}$ solution of \ref{5.1.s}-\ref{5.2.s}, we get that
		\begin{align}
			\begin{aligned}\label{5.9}
				&  C^{-1}_{p,\infty}\|D^su_s\|_{\bs L^p(\R^d)}\leq \|D^su_s\|_{\bs L^\infty(\R^d)}\leq g^*,\ 1\leq p<\infty,\\
				& \|u_s\|_{   C^{0,\beta}(\overline{\Omega})}\leq C_\beta,\ 
				\text{for any $\beta$, $0<\beta<s<1$}, 
			\end{aligned}
		\end{align}
		where the constant $C_\beta$ is independent of $s$.
		
		Letting
		$ \bs\Psi_s=\lambda_s D^s u_s$, arguing exactly as in \eqref{4.11}-\eqref{4.12} and \eqref{4.13}, by replacing the label $\nu$ by $s$, we obtain that
		\begin{align*}
			\|\lambda_s\|_{L^\infty(\R^d)}'\leq \tfrac{C_1}{g_*^2} \quad\text{and}\quad 
			\|\bs\Psi_s\|_{\bs L^\infty(\R^d)}'&\leq \tfrac{C_1 g^*}{g_*^2},
		\end{align*}
		where $C_1>0$ is a constant independent of $s$, $\sigma<s<1$
		
		Therefore, by compactness, in particular, by \eqref{5.9} and \eqref{camposembedding}, there are $u\in C^{0,\alpha}(\overline{\Omega})\cap\Lambda^{\sigma,p}(\Omega)$, for $0<\alpha<\beta$, $0<\sigma<1$ and $1<p<\infty$, $\bs\chi\in\bs L^\infty(\R^d)$, $\widetilde{\lambda}
		\in L^\infty(\R^d)'$, $\bs\Psi\in\bs L^\infty(\R^d)'$ and a generalised sequence $s\rightarrow 1$, such that
		\begin{align*}
			&	u_s\underset{s}\longrightarrow u\ \text{ in } \ \Lambda^{\sigma,p}_0(\Omega)\cap C^{0,\alpha}(\overline{\Omega}),\\
			&	D^su_s\underset{s}\lraup \bs\chi\ \text{ in } \ \bs L^\infty(\R^d)\text{-weak}^*,\\
			&	(\lambda_s)_\Omega\underset{s}\lraup \lambda \text{ in } \  L^\infty(\Omega)'\text{-weak}^*,\\
			&	\bs\Psi_s\underset{s}\lraup \bs\Psi\ \text{ in } \ \bs L^\infty(\R^d)\text{-weak}^*.
		\end{align*}
		
		Letting by $\widetilde{u}$ be the extension of $u$ by zero to $\R^d$, and applying Corollary \ref{2.3} componentwise to an arbitrarily $\bs\varphi\in C^\infty_c(\R^d)^d$, we have, by \eqref{wDs},
		\begin{equation*}
			\int_{\R^d}\bs \chi\cdot \bs\varphi =\lim_s\int_{\R^d}D^su^s\cdot\bs\varphi=-\lim_s\int_{\R^d}\widetilde{u}^sD^s\cdot \bs\varphi =-\int_{\R^d}\widetilde{u}D\cdot\bs\varphi,
		\end{equation*} 
		which means that $\bs\chi=D\widetilde{u}\in\bs L^\infty(\R^d)$ and $u\in W^{1,\infty}_0(\Omega)$, and therefore $D\widetilde{u}=\widetilde{Du}$.
		
		Arguing as in \eqref{pertence_convexo} with an arbitrarily measurable subset $w\subset\R^d$ with finite measure and with $\bs \xi=\frac{Du}{|Du|}\chi_V$, $V=w\cap\{|Du|\neq 0\}$, we obtain
		\begin{equation*}
			\int_w|D\widetilde{u}|=\int_{\R^d}D\widetilde{u}\cdot\bs\xi=\lim_s \int_{\R^d}D^s\widetilde{u}_s\cdot\bs\xi\leq \varlimsup_s\int_w|D^su_s|\leq\lim_s\int_wg_s=\int_wg_1
		\end{equation*}
		and so $|Du|\leq g_1$ a.e. in $\Omega$, i.e. $u\in \K_{g_1}$.
		
		Observe that we still have the lower semicontinuity property
		\begin{equation}
			\label{5.18}
			\varliminf_s\mathscr{L}^s(u_s,u_s)\geq \mathscr{L}^1(u,u)
		\end{equation}
		as we easily see by using \eqref{matrix} and taking $\varliminf_s$ in 
		\begin{equation*}
			\int_{\R^d}AD^su_s\cdot D^su_s\geq \int_{\R^d}AD^su_s\cdot D\widetilde{u}+\int_{\R^d}AD\widetilde{u}\cdot D^su_s-\int_{\R^d}AD\widetilde{u}\cdot D\widetilde{u}.
		\end{equation*}
		
		On the other hand, recalling \eqref{1sinfty}, we have $C^\infty_c(\Omega)\subset W^{1,\infty}_0(\Omega)\subset \Lambda^{s,\infty}_0(\Omega)$. Taking the limit $s\rightarrow 1$ in \ref{5.1.s} with $v\in C^\infty_c(\Omega)$ we get 
		\begin{equation}
			\label{5.19}
			\mathscr{L}^1(u,v)+\bsl \bs\Psi_\Omega,Dv\bsr=[F,v]_1.
		\end{equation}
		
		Since for each $v\in W^{1,\infty}_0(\Omega)$ we may take a sequence $v_n\in C^\infty_0(\Omega)$ such that $v_n\underset{n}{\longrightarrow}v$ in $H^1_0(\Omega)$ with $Dv_n\underset{n}{\longrightarrow}Dv$ in $\bs L^\infty(\Omega)\text{-weak}^*$, the equation \eqref{5.19} also holds for any $v\in W^{1,\infty}_0(\Omega)$, as $\bs\Psi_\Omega\in\bs L^\infty(\Omega)'$. So {\renewcommand\slab{1}\ref{5.1.s}} will follows if we show $\bs\Psi_\Omega=\lambda Du$, with $\lambda=\widetilde{\lambda}_\Omega$, which can be done exactly as in the proof of \eqref{psilambda}, by replacing the subscript $\nu$ by $s$ in \eqref{4.20} and in \eqref{4.24}.
		
		Similarly, the corresponding limit \eqref{4.25} as $s\rightarrow 1$ implies $\langle\lambda(g^2-|Du|^2) ,1\rangle=0$ and the same argument of \eqref{4.26} yields $\lambda(g-|Du|)=0$ in $L^\infty(\Omega)'$, showing that $(u,\lambda)$ also satisfies {\renewcommand\slab{1}\ref{5.2.s}}  in $\Omega$.
	\end{proof}
	\begin{remark}
		In the coercive case, i.e., if \eqref{Aelliptic} and \eqref{coercive} hold, it is clear that $u_s$ and $u_1$ are also the unique solutions of the respective variational inequalities \eqref{iv}, with $s\leq 1$. In this case, in particular, with $c=f_\#=0$ and $\bs b=\bs d=\bs 0$, the result was given in \cite{AzevedoRodriguesSantos2022} only with the convergence $u_s\underset{s}{\longrightarrow}u$ in $H^\sigma_0(\Omega)$, $0< \sigma<1$.
	\end{remark}
	
	\begin{remark}
		Under the assumptions of Theorem \ref{teo3.5} it is easy to obtain the estimates, similarly to \eqref{estimate},
		\begin{equation*}
			\|u_s\|_{H^s_0(\Omega)}\leq \tfrac{C_*}{\delta}\|f_*\|_{L^{2^\#}(\Omega)}+\tfrac{1}{\delta}\|\bs f\|_{\bs L^2(\R^d)}
		\end{equation*}
		and, denoting $\Gamma^s\in \partial I_{\K^s_{g_s}(u_s)}$, as in Remark \ref{rem3.4}, it is easy to conclude that $\Gamma^s$ is also uniformly bounded in $H^{-s}(\Omega)$, independently of $\sigma<s<1$. This allows us to take subsequences $u_s\underset{s}{\longrightarrow} u$ in $H^\sigma_0(\Omega)$ and $\Gamma^s\underset{s}{\longrightarrow}\Gamma$ in $H^{-\sigma}(\Omega)$, $\forall \sigma <1$, with $u\in H^1_0(\Omega)$, $\Gamma\in H^{-1}(\Omega)$ satisfying the local problem $s=1$.
	\end{remark}
	
	\begin{remark}
		As in Remark \ref{local_local_4},  in Theorem \ref{teo5.1} we may replace the assumptions on $g_s$ by the weaker assumption $g_s\in L^\infty_{loc}(\R^d)$, with positive lower bound in any compact set, and $\displaystyle\lim_{|x|\rightarrow\infty}g_s(x)|x|^{d+s}=\infty$ uniformly in $s$.
	\end{remark}

	
	%
	
	\section*{Acknowledgements}
	The authors wish to thank the referee's careful reading of this manuscript and his suggestions, which allowed the improvement of its final redaction. 
	
	The research of J. F. Rodrigues was partially done under the framework of the Project PTDC/MATPU/28686/2017 at CMAFcIO/ULisboa and A. Azevedo and L. Santos were partially financed by Portuguese Funds
	through FCT (Funda\c c\~ao para a Ci\^encia e a Tecnologia) within the Projects UIDB/00013/2020 and UIDP/00013/2020.
	

\end{document}